\documentclass[11pt]{amsart}
\usepackage{amssymb, amsthm, amsmath, mathrsfs, graphicx}
\usepackage[all,cmtip]{xy}
\numberwithin{equation}{section}
\newtheorem{theorem}{Theorem}[section]
\newtheorem{lemma}[theorem]{Lemma}
\newtheorem{corollary}[theorem]{Corollary}
\newtheorem{proposition}{Proposition}[section]

\theoremstyle{definition}
\newtheorem{definition}{Definition}

\theoremstyle{remark}
\newtheorem{remark}{Remark}

\newcommand{ \rn }[1]{\mathbb{R}^{#1}}

\title[An Inverse Source Problem with Partial Data]{An Inverse Source Problem in Radiative Transfer with Partial Data}
\date{October 26, 2011}
\author{Mark Hubenthal}
\address{Department of Mathematics, University of Washington, Seattle, WA 98195}
\email{hubenjm@math.washington.edu}

\begin{document}
\begin{abstract}The inverse source problem for the radiative transfer equation is considered, with partial data. Here it is shown that under certain smoothness conditions on the scattering and absorption coefficients, one can recover sources supported in a certain subset of the domain, which we call the visible set. Furthermore, it is shown for an open dense set of $C^{\infty}$ absorption and scattering coefficients that one can recover the part of the wave front set of the source that is supported in the microlocally visible set, modulo a function in the Sobolev space $H^{k}$ for $k$ arbitrarily large. This is an extension to the full data case, where the complete recovery of an arbitrary source has been shown.\end{abstract}
\maketitle
\section{Introduction}
We consider a problem relevant to optical molecular imaging (OMI), which is a fast-growing research area. In this application, biochemical markers can be used to detect the presence of specific molecules or genes, and suitably designed markers could potentially identify diseases before phenotypical symptoms even appear. The markers are typically light-emitting molecules, such as fluorophores or luminophores. In contrast to Single Positron Emission Computed Tomography (SPECT),  Positron Emission Tomography (PET), or Magnetic Resonance Imaging (MRI), optical markers emit low-energy near-infrared photons that are relatively harmless to human tissue. Further specifics can be found in the bioengineering literature such as \cite{be1,be2,be3}. 

The inverse problem we consider consists of reconstructing the spatial distribution of a radiation source from measurements of photon intensities at the boundary of the medium in specific outgoing directions. In many applications, the propagation of photons emitted can be modeled as inverse source problems of steady-state radiative transfer equations. Once we know the optical properties of the underlying medium, the problem of determining the source is feasible. It is shown in \cite{inversesource} that under mild assumptions on the scattering and absorption parameters of the medium this is possible. However, in the partial data case, which will be made more clear shortly, one can only hope to recover information about the singularities of the source. In particular, we seek to recover information about the wavefront set of the source function. We now describe more precisely the mathematical problem.

We assume $\Omega$ to be a bounded domain in $\rn{n}$ with smooth boundary and outer unit normal vector $\nu(x)$. As in \cite{inversesource}, we also assume that $\Omega$ is strictly convex. The authors of that paper note that this is not an essential assumption, since for the problem considered one can always push the boundary away and make it strictly convex without losing generality. Moreover, we assume that the data is given on the boundary of a larger domain $\Omega_{1} \Supset \Omega$. We remark that this condition is not needed for existence and uniqueness of the solution to the forward problem, but it is required for the stability result (\ref{stability}), which is adapted from the proof for the complete data case.

Consider the radiative transfer equation
\begin{align}
\theta \cdot \nabla_{x}u(x,\theta) + \sigma(x,\theta)u(x,\theta) - \int_{\mathbb{S}^{n-1}}k(x,\theta,\theta')u(x,\theta')\,d\theta' & = f(x), \notag\\
\quad u|_{\partial_{-}S\Omega} & = 0, \label{transport}
\end{align}
where the absorption $\sigma$ and the collision kernel $k$ are functions with regularity specified later, the solution $u(x,\theta)$ gives the intensity of photons at $x$ moving in the direction $\theta$, and $\partial_{\pm}S\Omega$ is the set of points $(x,\theta) \in \partial \Omega \times \mathbb{S}^{n-1}$ such that $\pm \nu(x) \cdot \theta > 0$. That is, $\partial_{\pm}S\Omega$ is the set of points $(x,\theta) \in \partial \Omega \times \mathbb{S}^{n-1}$ such that $\theta$ is pointing outward or inward, respectively. The source term $f$ will be assumed to depend on $x$ only for our purposes. We also remark that equation \ref{transport} is only applicable at a single frequency, as the parameters $\sigma$ and $k$ typically depend highly on frequency. In particular, for high energy photons there is a coupling between energy and angle, whereas for photons with low energy scattering is not accompanied by an energy change.

In the case of full data, we have boundary measurements
\begin{equation}
Xf(x,\theta) = u\vert_{\partial_{+}S\Omega}.
\end{equation}
In \cite{inversesource}, it is shown that for an open, dense set of absorption and scattering coefficients $(\sigma, k) \in C^{2}(\overline{\Omega} \times \mathbb{S}^{n-1}) \times C^{2}(\overline{\Omega} \times \mathbb{S}^{n-1} \times \mathbb{S}^{n-1})$, one can recover $f \in L^{2}(\Omega)$ uniquely from boundary measurements $Xf$ on all of $\partial_{+}S\Omega$. To set up the case of partial data, first let $V \subset \partial_{+}S\Omega$ be open and let $\widetilde{V} \Subset V$. Let $\chi_{V}\in C_{0}^{\infty}(\partial_{+}S\Omega)$ be a smooth cutoff function such that $\chi_{V}(x,\theta) \equiv 1$ for $(x,\theta) \in \widetilde{V}$ and $\chi_{V}(x,\theta) \equiv 0$ for $(x,\theta) \notin V$. The boundary measurements are then given by 
\begin{equation}
X_{V}f(x,\theta) = \chi_{V}(x,\theta) u\vert_{\partial_{+}S\Omega}.
\end{equation}
To make notation a bit simpler, if $V = \partial_{+}S\Omega$ (complete data) we will just write $X$, since in this case $X_{V} = X$.

In section \ref{direct} we will review the direct problem and some relevant results  of use in the partial data case. We also establish some results about singular integrals that will be needed to prove the main theorem.

In section \ref{inverse} we consider the inverse problem with partial data, which consists of determining the source term $f$ from measuring $X_{V}f$. We also compute the normal operator $X_{V}^{*}X_{V}$ when the scattering coefficient $k= 0$. Note that when $\sigma = k = 0$, the operators $X$ and $X_{V}$ are the standard X-ray transforms with full and limited data, respectively. When $k=0$, then $X_{V}$ is a weighted X-ray transform.

Following, in section \ref{partialinjective} we establish an injectivity result for $f \in L^{2}(\Omega)$ supported in the \textit{visible set} assuming analytic $\sigma$.

Finally, in section \ref{partialwavefront} we prove the main theorem that one can recover the visible singularities of $f$ with respect to the chosen set $V$ when $\sigma \in C^{\infty}(\Omega \times \mathbb{S}^{n-1})$. Both results are based on the microlocal approach used in \cite{xraygeneric}. Results needed pertaining to singular integral operators are located in appendix \ref{singularresults}.

\section{Statement of Main Results}
When dealing with the inverse problem, which we will describe in detail in section \ref{inverse}, we need to take a larger domain (strictly convex for convenience) that compactly contains $\Omega$. That is, fix a strictly convex open set $\Omega_{1}$ with smooth boundary such that $\Omega_{1} \Supset \Omega$. The strict convexity of $\Omega_{1}$ ensures that the functions $\tau_{\pm}(x,\theta)$ are smooth, where $\tau_{\pm}(x,\theta)$ is the travel time from $x \in \Omega_{1}$ to $\partial \Omega_{1}$ in the direction $\pm \theta$. In other words
\begin{equation}
(x + \tau_{\pm}(x,\theta)\theta, \theta) \in \partial_{\pm}S\Omega_{1}.
\end{equation}
We will extend $\sigma$ and $k$ to functions on $\Omega_{1}$ with the same regularity. We choose and fix this extension as a continuous operator in those spaces. Now define $X_{1}:L^{2}(\Omega_{1}) \to L^{2}(\partial_{+}S\Omega_{1})$ in the same way as for $X$. From this we can look at the restriction of $X_{1}$ applied to functions $f$ supported in $\overline{\Omega}$ by first extending such $f$ as zero on $\Omega_{1} \setminus \Omega$. Essentially, we are moving the observation surface outward a bit and taking measurements on $\partial \Omega_{1}$. When dealing with the inverse problem, we will usually abuse notation and write $X$ instead of $X_{1}$, with the understanding that we've already extended the domain $\Omega$ to $\Omega_{1}$.


It is proven in (Theorem 2, \cite{inversesource}) that the operator $X$ is injective for such an open, dense set of coefficients $(\sigma, k)$ as in (Theorem 1, \cite{inversesource}) with $f \in L^{2}(\Omega)$, and a stability result is obtained for the normal operator $X^{*}X:L^{2}(\Omega) \to L^{2}(\Omega)$. Here the adjoint $X^{*}:L^{2}(\partial_{+}S\Omega, d\Sigma) \to L^{2}(\Omega \times \mathbb{S}^{n-1})$ is defined with respect to the measure $d\Sigma$, which we define shortly. More specifically, for an open and dense set of pairs $(\sigma, k) \in C^{2}(\overline{\Omega} \times \mathbb{S}^{n-1}) \times C^{2}(\overline{\Omega}_{x} \times \mathbb{S}_{\theta'}^{n-1}; C^{n+1}(\mathbb{S}_{\theta}^{n-1}))$, including a neighborhood of $(0,0)$, we have that the conclusions of (Theorem 1, \cite{inversesource}) hold in $\Omega_{1}$, that $X_{1}$ is injective on $L^{2}(\Omega)$, and the stability estimate $\|f\|_{L^{2}(\Omega)} \leq C\|X_{1}^{*}X_{1}f \|_{H^{1}(\Omega_{1})}$ for a constant $C > 0$ locally uniform in $(\sigma, k)$.

Before stating the main results of this paper, we need to define the set of points such that $X_{V}$ is injective when restricted to sources supported there. This set will clearly depend on $V$. We also denote by $l_{x,\theta}(t)$ as the line starting at $x$ with direction $\theta$.
\begin{definition}We define the \textit{visible set} $\mathcal{M} \subset \Omega$ by
 \begin{align}
\mathcal{M} & = \{x \in \Omega \, | \, \forall \theta \in \mathbb{S}^{n-1}\, \exists (z, \theta^{\perp}) \in V \textrm{with }\theta^{\perp} \cdot \theta = 0 \notag \\
& \quad \textrm{ such that } l_{z,\theta^{\perp}} \textrm{ intersects } x\} \label{visibleset}.
\end{align}
\end{definition}
It is relatively straightforward to show since $V$ is open, $\mathcal{M}$ is open as well. The proof is left to the reader.
\begin{figure}
\centering
\includegraphics[width=0.5\textwidth]{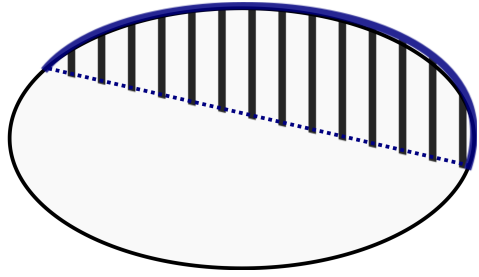}
\caption{The visible set is shaded here in the case where $V$ consists of a connected open subset of the boundary together with all outgoing directions.}
\end{figure}
Now we can state the first main result, which is an injectivity condition adapted from results in \cite{xraygeneric}. 
\begin{theorem}Let $V \in \partial_{+}S\Omega_{1}$ be an open set and let $\mathcal{M}$ be as defined above. Let $W \Subset \mathcal{M}$. Then there exists an open and dense set of pairs
\begin{equation}
(\sigma, k) \in C^{2}(\overline{\Omega} \times \mathbb{S}^{n-1}) \times C^{2}(\overline{\Omega}_{x} \times \mathbb{S}_{\theta'}^{n-1}; C^{n+1}(\mathbb{S}_{\theta}^{n-1})), \label{regularity}
\end{equation}
including a neighborhood of $(0,0)$, such that for each $(\sigma, k)$ in that set, the direct problem (\ref{transport}) has a unique solution $u \in L^{2}(\Omega_{1} \times \mathbb{S}^{n-1})$ for any $f \in L^{2}(\Omega \times \mathbb{S}^{n-1})$, $X_{V}$ extends to a bounded operator from $L^{2}(\Omega_{1} \times \mathbb{S}^{n-1})$ to $L^{2}(\partial_{+}S\Omega_{1}, d\Sigma)$, and
\begin{enumerate}
\item the map $X_{V}$ is injective on $L^{2}(W)$,
\item the following stability estimate holds
\begin{equation}
\|f\|_{L^{2}(\Omega)} \leq C \|X_{V}^{*}X_{V}f \|_{H^{1}(\Omega_{1})}, \quad \forall f \in L^{2}(W), \label{stability}
\end{equation}
\end{enumerate}
with a constant $C > 0$ locally uniform in $(\sigma, k)$. \label{injectivetheorem}
\end{theorem}
\begin{remark}
The proof of uniqueness and stability for the direct problem (\ref{transport}) as stated in Theorem \ref{injectivetheorem} is essentially the same as the one contained in \cite{inversesource}, so we will focus on the subtle differences. Furthermore, the proof that $X_{V} = \chi_{V}X$ extends to a bounded operator from $L^{2}(\Omega_{1})$ to $L^{2}(\partial_{+}S\Omega_{1}, d\Sigma)$ follows from the proof in \cite{inversesource} that $X$ is bounded and the fact that multiplication by $\chi_{V}$ is bounded on $L^{2}(\partial_{+} S\Omega_{1}, d\Sigma)$.
\end{remark}

For sources $f$ with more general supports, we hope to be able to recover certain covectors in the wavefront set of $f$. Those covectors $(x,\xi) \in T^{*}\Omega$ that can be detected will depend on $V$ in the following way:
\begin{definition}The \textit{microlocally visible set} corresponding to partial measurements on $\partial_{+}S\Omega_{1}$ is given by
 \begin{equation}
\mathcal{M}' := \{ (x,\xi) \in T^{*}\Omega \, | \, \exists \theta \in \mathbb{S}^{n-1} \textrm{ such that } \theta \cdot \xi = 0 \textrm{ and } \chi_{V}^{\#}(x,\theta) \neq 0 \}.\label{microvisibleset}
\end{equation}
\end{definition}
Here $\chi_{V}^{\#}(x,\theta)$ is the extension of $\chi_{V}:\partial_{+}S\Omega_{1} \to \mathbb{R}$ to $\Omega_{1} \times \mathbb{S}^{n-1}$ defined by $\chi_{V}^{\#}(x,\theta) = \chi_{V}(x + \tau_{+}(x,\theta)\theta, \theta)$.

\begin{figure}
\centering
\includegraphics[width=0.5\textwidth]{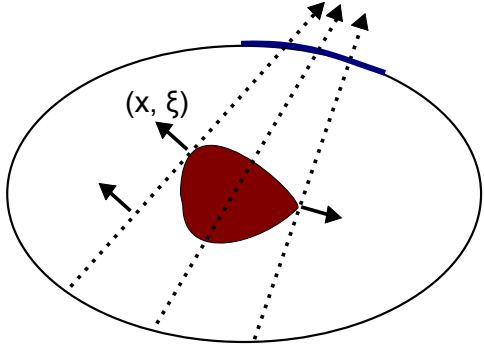}
\caption{An example where $(x,\xi)$ is in the microlocally visible set $\mathcal{M}'$, given that the source $f$ is the characteristic function of the shaded set.}
\end{figure}

\begin{theorem}Let $l$ be a positive integer. There exists an open dense set $\mathcal{O}_{l}$ of pairs $(\sigma, k) \in  C^{\infty}(\overline{\Omega} \times \mathbb{S}^{n-1}) \times  C^{\infty}(\overline{\Omega}_{x} \times \mathbb{S}_{\theta'}^{n-1} \times \mathbb{S}_{\theta}^{n-1})$ depending on $l$ such that given $(\sigma, k) \in \mathcal{O}_{l}$,  if $(z,\xi) \in \mathcal{M}'$, then there exists a function $v \in H^{l}(\Omega)$ such that
\begin{equation}
(z,\xi) \notin \textrm{WF}(X_{V}^{*}X_{V}f) \Longrightarrow (z,\xi) \notin  \textrm{WF}(f+v).
\end{equation}
\label{microtheorem}
\end{theorem}

\section{The Direct Problem \label{direct}}
For notational convenience and to be consistent with convention, we set
\begin{equation}
T_{0} = \theta \cdot \nabla_{x}, \quad T_{1} = T_{0} + \sigma, \quad T = T_{0} + \sigma - K,
\end{equation}
where $\sigma$ denotes the operation of multiplication by $\sigma(x,\theta)$, and $K$ is defined by
\begin{equation}
Kf(x,\theta) = \int_{\mathbb{S}^{n-1}}k(x,\theta,\theta')f(x,\theta')\,d\theta'.
\end{equation}

If $k=0$, we have that
\begin{equation}
Xf(x,\theta) = I_{\sigma}f(x,\theta) := \int_{\tau_{-}(x,\theta)}^{0}E(x+t\theta,\theta)f(x+t\theta)\,dt, \quad (x,\theta) \in \partial_{+}S\Omega,
\end{equation}
where $\tau_{\pm}(x,\theta)$ is the arrival time defined by $(x + \tau_{\pm}(x,\theta)\theta,\theta) \in \partial_{\pm}S\Omega$ for $(x,\theta) \in \Omega \times \mathbb{S}^{n-1}$.
Here $E$ is defined by
\begin{equation}
E(x,\theta) = \exp\left( -\int_{0}^{\infty}\sigma(x + s\theta, \theta)\,ds \right).
\end{equation}
Note that if $\sigma > 0$ depends on $x$ only, then $I_{\sigma}$ is just the attenuated X-ray transform along the line through $x$ in the direction $\theta$. Moreover, in this case it is injective and \cite{novikov} gives an explicit inversion formula.

In the general case with $k \neq 0$, it is shown in (Theorem 1, \cite{inversesource}) that the direct problem (\ref{transport}) is well-posed even for $f$ depending on $x$ and $\theta$. That is, for an open and dense set of pairs 
\begin{equation*}
(\sigma, k) \in C^{2}(\overline{\Omega} \times \mathbb{S}^{n-1}) \times C^{2}(\overline{\Omega}_{x} \times \mathbb{S}_{\theta'}^{n-1}; C^{n+1}(\mathbb{S}_{\theta}^{n-1})),
\end{equation*}
including a neighborhood of $(0,0)$, the direct problem $Tu = f$ with $u \vert_{\partial_{-}S\Omega} = 0$ has a unique solution $u \in L^{2}(\Omega \times \mathbb{S}^{n-1})$ for any $f \in L^{2}(\Omega \times \mathbb{S}^{n-1})$ depending on both $x$ and $\theta$. Furthermore, the complete data operator $X$, which is only a priori bounded when restricted to sufficiently smooth $f$, extends to a bounded operator
\begin{equation*}
X: L^{2}(\Omega \times \mathbb{S}^{n-1}) \to L^{2}(\partial_{+}S\Omega, d\Sigma).
\end{equation*}
The proof of this relies on using the fact that
\begin{equation}
[T_{1}^{-1}f](x,\theta) = \int_{-\infty}^{0}\exp\left(  -\int_{s}^{0}\sigma(x + \tau \theta, \theta)\,d\tau \right)f(x + s\theta,\theta)\,ds, \label{t1inverse}
\end{equation}
as well as Fredholm Theory applied to the resolvent $(\textrm{Id} - T_{1}^{-1}K)^{-1}$.

In order to solve $Tu = f$, we observe that $Tu = T_{1}u - Ku = f$, and so applying $T_{1}^{-1}$ to both sides yields
\begin{equation}
u = T_{1}^{-1}(Ku + f).
\end{equation}
This is equivalent to the integral equation
\begin{equation}
(\textrm{Id} - T_{1}^{-1}K)u = T_{1}^{-1}f.
\end{equation}
Thus, if $\textrm{Id} - T_{1}^{-1}K$ is invertible, we can solve the forward problem uniquely for 
\begin{equation}
u = T^{-1}f = (\textrm{Id} - T_{1}^{-1}K)^{-1}T_{1}^{-1}f.
\end{equation}
To find $k$ such that $T^{-1}$ exists, we note that $(\textrm{Id} - T_{1}^{-1}K)^{-1}T_{1}^{-1} = T_{1}^{-1}(\textrm{Id} - KT_{1}^{-1})^{-1}$ and look at the operator
\begin{equation}
A(\lambda) = \left( \textrm{Id} - (\lambda KT_{1}^{-1})^{2}\right)^{-1} \label{resolvent2}
\end{equation}
It is shown in \cite{inversesource} that the operator $\left(KT_{1}^{-1}\right)^{2}$ is compact, and for $\lambda = 0$ the resolvent (\ref{resolvent2}) exists. By the analytic Fredholm theorem (Theorem VI.14, \cite{simon1}), we have that $A(\lambda)$ is a meromorphic family of bounded operators with poles contained in a discrete set. It can be shown that
\begin{equation}
\left( \textrm{Id} - \lambda KT_{1}^{-1}\right)^{-1} = \left( \textrm{Id} + \lambda KT_{1}^{-1}\right) A(\lambda). \label{resolvent1}
\end{equation}
In particular, the r.h.s above is a easily seen to be a right inverse. To show that it is a left inverse as well, we can expand $A(\lambda)$ as a Neumann series for $\|KT_{1}^{-1}\| \ll 1$ and then use analytic continuation to show that it remains true for all $\lambda$ that are not poles of $A(\lambda)$. These ideas will be useful later when proving Theorem \ref{injectivetheorem}.

\section{The Inverse Source Problem with Partial Data \label{inverse}}
Let $V \subset \partial_{+} S\Omega$ be some open subset. Then the boundary measurements for the problem (\ref{transport}) with partial data are modelled by
\begin{equation}
X_{V}f(x,\theta) := \chi_{V}(x,\theta)u \vert_{\partial_{+}S\Omega}, \quad (x,\theta) \in \partial_{+}S\Omega
\end{equation}
where $u(x,\theta)$ is a solution of (\ref{transport}), and $\chi_{V}:\partial_{+}S\Omega \to [0,1]$ is a smooth function equal to $0$ for $(x,\theta) \notin V$ and $\chi_{V}(x,\theta) = 1$ for $(x,\theta) \in \widetilde{V} \Subset  V$ for some open $\widetilde{V}$. We also define the operator $J:L^{2}(\Omega) \to L^{2}(\Omega \times \mathbb{S}^{n-1})$ by
\begin{equation*}
Jf(x,\theta) = f(x).
\end{equation*}

If $k=0$, we have that
\begin{equation}
X_{V}f(x,\theta) = I_{\sigma, V}f(x,\theta) := \chi_{V}(x,\theta) I_{\sigma}f(x,\theta).
\end{equation}
We will proceed as in \cite{inversesource} by looking at $X_{V}$ as a perturbation of $I_{\sigma, V}$. Wishful thinking suggests that $X_{V}^{*}X_{V}$ is a relatively compact perturbation of $I_{\sigma, V}^{*} I_{\sigma, V}$, the normal operator corresponding to $k=0$. Here $X_{V}^{*}$ is the adjoint of $X_{V}$ with respect to the measure $d\Sigma$ on $\partial_{+}S\Omega$ given by
\begin{equation}
d\Sigma = |\theta \cdot \nu(x)|\,dS_{x}\,dS_{\theta},
\end{equation}
where as stated earlier, $\nu(x)$ is the outward unit normal to the boundary $\partial \Omega$.

First let us consider the case when $k=0$ and compute $I_{\sigma,V}^{*}$. Note that $I_{\sigma,V}:L^{2}(\Omega \times S^{n-1})\to L^{2}(\partial_{+}S\Omega, d\Sigma)$, and hence $I_{\sigma,V}^{*}:L^{2}(\partial_{+}S\Omega, d\Sigma) \to L^{2}(\Omega \times \mathbb{S}^{n-1}).$ For now we will restrict ourselves to applying $I_{\sigma,V}$ to functions $f$ that depend on $x$ only. Given $h(x,\theta) \in L^{2}(\partial_{+}S\Omega, d\Sigma)$ and $f(x)\in L^{2}(\Omega)$, one can show that
\begin{equation*}
\langle I_{\sigma, V}^{*}h(x), f(x) \rangle_{L^{2}(\Omega \times \mathbb{S}^{n-1})} = \int_{\Omega} \int_{\mathbb{S}^{n-1}}h^{\#}(y,\theta)\chi_{V}^{\#}(y,\theta)\overline{E}(y,\theta)\overline{f}(y)\,d\theta\, dy,
\end{equation*}
where $g^{\#}(x,\theta)$ is the extension of $g:\partial_{+}S\Omega \to \mathbb{R}$ to $\Omega \times \mathbb{S}^{n-1}$ defined by $g^{\#}(x,\theta) = g(x + \tau_{+}(x,\theta)\theta, \theta)$.
We also made use of the diffeomorphism $\phi:\partial_{+}S\Omega \times \mathcal{O} \to \Omega \times \mathbb{S}^{n-1}$ where $\mathcal{O} = \{(\tau_{-}(x,\theta), 0)\, | \, (x,\theta) \in \partial_{+}S\Omega \}$. This map is defined by $\phi(x,\theta,t) = (x+t\theta,\theta)$. The Jacobian determinant of $\phi$ is $|\nu(x)\cdot \theta|$; see (Lemma 2.1, \cite{choulli1}). Note that $\phi^{-1}:\Omega \times \mathbb{S}^{n-1} \to \partial_{+}S\Omega \times \mathcal{O}$ is given by $\phi^{-1}(x,\theta) = (x+\tau_{+}(x,\theta)\theta, \theta, \tau_{+}(x,\theta))$. Hence the adjoint in the no-scattering case has the equation
\begin{equation}
I_{\sigma,V}^{*}h(x,\theta) = \int_{S^{n-1}}\overline{E}(x,\theta)h^{\#}(x,\theta)\chi_{V}^{\#}(x,\theta)\,d\theta.
\end{equation}

\subsection{The Normal Operator $I_{\sigma, V}^{*}I_{\sigma,V}$}
Similar to the way in which we derived the adjoint operator $I_{\sigma, V}^{*}$, we may compute the normal operator $I_{\sigma,V}^{*}I_{\sigma,V}: L^{2}(\Omega) \to L^{2}(\Omega)$ as
\begin{align}
& \langle I_{\sigma,V}^{*}I_{\sigma,V}f(x),g(x)\rangle_{L^{2}(\Omega)} \notag \\
& = \int_{\Omega}\left[ \int_{S^{n-1}}\overline{E}(x,\theta)\chi_{V}^{\#}(x,\theta) \left( I_{\sigma,V}f(x,\theta) \right)^{\#}\,d\theta \right] \overline{g}(x)\,dx \notag \\
& = \int_{\Omega}\left[ \int_{S^{n-1}}\overline{E}(x,\theta)\chi_{V}^{\#}(x,\theta) \left( \chi_{V}(x,\theta)\int_{\mathbb{R}}E(x+t\theta,\theta)f(x+t\theta)\,dt \right)^{\#} \right] \overline{g}(x)\,d\theta\, dx \label{normaloperator}\\
& = \int_{\Omega}\int_{\Omega}\frac{\overline{E}(x,\frac{y-x}{|y-x|})\left|\chi_{V}^{\#}(x,\frac{y-x}{|y-x|})\right|^{2}E(y,\frac{y-x}{|y-x|})f(y)}{|y-x|^{n-1}} \overline{g}(x)\,dy\, dx. \notag
\end{align}
In the last line we used the substitution $y = x + t \theta$ to convert the integral over $\mathbb{S}^{n-1}\times \mathbb{R}$ in $(\theta,t)$ to an integral over $\Omega$ in $y$. Thus
\begin{equation}
I_{\sigma,V}^{*}I_{\sigma,V}f(x) = \int_{\Omega}\frac{ E(y,\frac{y-x}{|y-x|})\overline{E}(x, \frac{y-x}{|y-x|}) \left| \chi_{V}^{\#}(x, \frac{y-x}{|y-x|}) \right|^{2} f(y)}{|y-x|^{n-1}} \,dy.
\end{equation}

In the case that $\sigma$ is $C^{\infty}$, we would like to know where $I_{\sigma, V}^{*}I_{\sigma,V}$ is elliptic. One can show using (Theorem 3.4, \cite{grigis}) that
\begin{equation}
I_{\sigma,V}^{*}I_{\sigma,V}f(x)  = (2\pi)^{-n}\int e^{i(x-y)\cdot \xi}b(x,\xi)f(y)\,dy\, d\xi
\end{equation}
where
\begin{equation}
b(x,\xi) = (2\pi)^{-n} \int e^{-i(x-y)\cdot \xi}\frac{ E(y,\frac{y-x}{|y-x|})\overline{E}(x, \frac{y-x}{|y-x|}) \left| \chi_{V}^{\#}(x, \frac{y-x}{|y-x|}) \right|^{2} }{|y-x|^{n-1}}\,dy.\label{normalsymbol}
\end{equation}
We can now see that for $\sigma \in C^{\infty}(\Omega \times \mathbb{S}^{n-1})$ $I_{\sigma,V}^{*}I_{\sigma,V}$ is a pseudodifferential operator of order $-1$, since its kernel is weakly singular. More specifically, Proposition 1 of \cite{inversesource} gives that $I_{\sigma,V}^{*}I_{\sigma,V}:L^{2}(\Omega) \to H^{1}(\Omega)$. See also \cite{stein}.

Now, unfortunately equation (\ref{normalsymbol}) isn't particularly useful when trying to determine where $b(x,\xi)$ is elliptic. But recall (\ref{normaloperator}), which shows that
\begin{equation}
I_{\sigma,V}^{*}I_{\sigma,V}f(x) =  \int_{\mathbb{S}^{n-1}}\int_{\mathbb{R}}A(x,t,\theta)f(x+t\theta)\,dt\,d\theta,
\end{equation}
for a particular function $A$. By Lemma 4.2 of \cite{xraygeneric} we have that if $A \in C^{\infty}(\Omega_{1} \times \mathbb{R} \times \mathbb{S}^{n-1})$ (which occurs if $\sigma \in C^{\infty}(\Omega_{1} \times \mathbb{S}^{n-1})$), then $I_{\sigma,V}^{*}I_{\sigma,V}$ is a classical $\Psi$DO of order $-1$ with full symbol
\begin{equation}
b(x,\xi) \sim \sum_{m=0}^{\infty}b_{m}(x,\xi),
\end{equation}
where
\begin{equation}
b_{m}(x,\xi) = 2\pi \frac{i^{m}}{m!}\int_{\mathbb{S}^{n-1}}\partial_{t}^{m}A(x,0,\theta)\delta^{(m)}(\theta \cdot \xi)\,d\theta.
\end{equation}

To check for ellipticity, we need only look at the principal symbol corresponding to when $m=0$. This is just
\begin{equation}
b_{0}(x,\xi) = 2\pi \int_{\theta \cdot \xi = 0}|E(x,\theta)|^{2}|\chi_{V}^{\#}(x,\theta)|^{2}\,dS(\theta). \label{principalsymbol}
\end{equation}
Since $E$ is nonvanishing, we immediately have by (\ref{principalsymbol}) that $b(x,\xi)$ is elliptic on the set $\mathcal{M}'$.

\section{Injectivity of $X_{V}$ Restricted to the Visible Set \label{partialinjective}}

\subsection{Injectivity of $I_{\sigma,V}$ and $I_{\sigma,V}^{*}I_{\sigma,V}$} \label{injectivity}
Since we are only able to access some open subset $V$ of $\partial_{+} S\Omega_{1}$, we cannot expect for the operator $I_{\sigma, V}$ or the normal operator $I_{\sigma,V}^{*}I_{\sigma,V}$ to be injective. However, from \cite{xraygeneric} we can obtain injectivity for sources $f$ supported in a particular subset of $\Omega$. But first we must introduce the notion of a regular family of curves. We will use the notation $l_{x,\theta}$ to denote the line segment through $x \in \Omega$ in the direction $\theta \in \mathbb{S}^{n-1}$ with endpoints on $\partial \Omega_{1}$. We can also assume that $l_{x,\theta}(0) = x$ and $l_{x,\theta}'(0) = \theta$. It is also clear that the lines $l_{x,\theta}$ depend smoothly on $(x, \theta) \in T\Omega$ in the sense that the function $l(x,\theta,t) = l_{x,\theta}(t)$ depends smoothly on $x$, $\theta$ and $t$ separately. In fact, we have $l(x,\xi,t) = x + t \xi$ where $t \in (a(x,\xi), b(x,\xi))$, an interval containing $0$, and $l(x,\xi,a(x,\xi)), l(x,\xi,b(x,\xi)) \in \partial \Omega_{1}$.

\begin{definition}Let $\Gamma$ be an open family of smooth (oriented) curves on $\Omega$, with a fixed parametrization on each one of them, with endpoints on $\partial \Omega$, such that for each $(x,\xi) \in T\Omega \setminus 0$, there is at most one curve $\gamma_{x,\xi} \in \Gamma$ through $x$ in the direction $\xi$, and the dependence on $(x,\xi)$ is smooth. We say that $\Gamma$ is a \textit{regular} family of curves, if for any $(x,\xi) \in T^{*}\Omega$, there exists $\gamma \in \Gamma$ through $x$ normal to $\xi$ without conjugate points.\end{definition}

\begin{remark}In our specific case, all curves taken are straight lines and have no conjugate points. Moreover, if we let $\Gamma_{\mathcal{M}}$ be the set of lines in $\Omega_{1}$ which intersect $\mathcal{M}$, then it turns out (as is shown in the proof of Theorem \ref{Iinjective}) that $\Gamma_{\mathcal{M}}$ is a regular family when restricted to $\mathcal{M}$. This is the motivation behind how $\mathcal{M}$ was defined in the first place.\end{remark}

\begin{theorem}Let $\sigma$ be analytic on $\Omega$. If $I_{\sigma,V}f = 0$ for $f \in \mathcal{D}'(\Omega_{1})$ supported in $W \Subset \mathcal{M}^{\textrm{int}}$, then $f = 0$. In particular, $I_{\sigma,V}$ is injective on $L^{1}(W)$.\label{Iinjective}\end{theorem}

\begin{proof}It is clear that the collection $\Gamma$ of lines in $\Omega_{1}$ is an analytic regular family of curves (see \cite{xraygeneric}). Let $\Gamma_{\mathcal{M}}$ be only those lines which pass through $\mathcal{M}$. We claim that $\Gamma_{\mathcal{M}}$ is a regular family of curves when restricted to $\mathcal{M}$. To see this, let $x \in \mathcal{M}$ and $\theta \in S^{n-1}$. Then by definition of $\mathcal{M}$ there exists $z \in V$ and an angle $\theta^{\perp}$ normal to $\theta$ such that the line $l_{z,\theta^{\perp}}$ passes through $x$ and $(z,\theta^{\perp}) \in \partial_{+}S\Omega$. By definition we also have that $l_{z,\theta^{\perp}} \in \Gamma_{\mathcal{M}}$, which proves the claim. Now, by Theorem 1 of \cite{xraygeneric} we have that $f$ is analytic on $\mathcal{M}$ with support properly contained in $\mathcal{M}^{\textrm{int}}$. In particular, $f = 0$ on an open subset of each component of $\mathcal{M}$.  Therefore $f = 0$.\end{proof}

Although the definition of $\mathcal{M}$ is a bit cryptic and difficult to visualize, it is possible to easily visualize an important subset of $\mathcal{M}$ when $V$ has a certain form, as shown by Lemma \ref{visiblesubset}. Here we use the notation $\mathrm{ch }{A}$ to denote the closed convex hull of a set $A \subset \rn{n}$.

\begin{lemma}Suppose that $V = \pi^{-1}(W)$ where $W$ consists of a countable collection of disjoint connected open subsets of $\partial \Omega_{1}$, and $\pi:\partial_{+}S\Omega_{1} \to \partial \Omega_{1}$ is the natural projection. Then $\bigcup_{j} \left(\mathrm{ch }{ W_{j} }\right)^{\mathrm{int}} \subset \mathcal{M}$ where $W_{j}$ is a given component of $W$.\label{visiblesubset}\end{lemma}

\begin{proof}Suppose $W = \bigcup_{\alpha} W_{\alpha}$ where $W_{\alpha}\subset \partial \Omega$ are disjoint connected open sets. Let $x \in \bigcup_{\alpha} \left(\mathrm{ch}{ W_{\alpha}}\right)^{\mathrm{int}}$. Let $\theta \in S^{n-1}$ and let $\theta^{\perp}$ be any vector perpendicular to $\theta$. If we consider that 
\begin{equation*}
\mathrm{ch}{ W_{\alpha} } = \overline{\{ \mathrm{hyperplanes  } \,  P \subset \rn{n} \, | \, P \cap W_{\alpha} = \emptyset \}^{\mathrm{c}}},
\end{equation*}
then $l_{x,\theta^{\perp}}$ must intersect $W_{\alpha}$ at some point $z$. Changing the direction of $\theta^{\perp}$ if necessary and using the strict convexity of $\Omega_{1}$, we have that $(z,\theta^{\perp}) \in \partial_{+}S\Omega_{1}$. This proves that $\left(\mathrm{ch}{ W_{\alpha} }\right)^{\mathrm{int}} \subset \mathcal{M}$ for all $\alpha$.\end{proof}

\subsection{Computing $X_{V}$ as a perturbation of $I_{\sigma,V}$ for $k\neq 0$}

In order to approach the case that $k \neq 0$, we will compute explicitly how $X_{V}$ differs from $I_{\sigma,V}$. Note that
\begin{equation}
Xf = \chi_{V}R_{+}T^{-1}f = \chi_{V}R_{+}(\textrm{Id} - T_{1}^{-1}K)^{-1}T_{1}^{-1}f,
\end{equation}
where
\begin{equation*}
R_{+}h = h \vert_{\partial_{+}S\Omega}.
\end{equation*}
If $f$ depends on $x$ only (the case we are primarily interested in), then
\begin{equation}
X_{V}f = \chi_{V}R_{+}T^{-1}Jf = \chi_{V}R_{+}(\textrm{Id} - T_{1}^{-1}K)^{-1}T_{1}^{-1}Jf.
\end{equation}
Now consider the identity
\begin{equation}
(\textrm{Id} - T_{1}^{-1}K)^{-1}T_{1}^{-1} = T_{1}^{-1}(\textrm{Id} - K T_{1}^{-1})^{-1},
\end{equation}
which implies that
\begin{equation}
X_{V}f = \chi_{V}R_{+}T_{1}^{-1}(\textrm{Id} - KT_{1}^{-1})^{-1}Jf.
\end{equation}

Writing $X_{V} = I_{\sigma,V} + L_{V}$ and noting that
\begin{equation*}
I_{\sigma,V} f = \chi_{V} R_{+}T_{1}^{-1}Jf,
\end{equation*}
we have that
\begin{equation}
X_{V} = I_{\sigma,V} + \chi_{V} R_{+}(-\textrm{Id} + (\textrm{Id} - T_{1}^{-1}K)^{-1})T_{1}^{-1}J.
\end{equation}
and so we have
\begin{equation}
L_{V} :=  \chi_{V} R_{+}(-\textrm{Id} + (\textrm{Id} - T_{1}^{-1}K)^{-1})T_{1}^{-1}J. \label{LV}
\end{equation}
Furthermore,
\begin{equation}
X_{V}^{*}X_{V} = I_{\sigma, V}^{*}I_{\sigma,V} + \mathcal{L}_{V}, \quad \mathcal{L}_{V} := I_{\sigma,V}^{*}L_{V} + L_{V}^{*}I_{\sigma,V} + L_{V}^{*}L_{V}.
\end{equation}

\begin{lemma} The operators
$$ \partial_{x}I_{\sigma,V}^{*}L_{V}, \quad \partial_{x}L_{V}^{*}I_{\sigma,V}, \quad \partial_{x}L_{V}^{*}L_{V}$$
are compact as operators mapping $L^{2}(\Omega_{1})$ into $L^{2}(\Omega_{1})$. \label{compact}
\end{lemma}

\begin{proof}Following the steps of the proof of Lemma 3 in \cite{inversesource}, first note that
\begin{equation}
(-\textrm{Id} + (\textrm{Id} - T_{1}^{-1}K)^{-1})T_{1}^{-1} = T_{1}^{-1}KT_{1}^{-1}(\textrm{Id} - KT_{1}^{-1})^{-1}.
\end{equation}
To prove it, we note that
\begin{equation*}
T_{1}^{-1}KT_{1}^{-1} =  (-\textrm{Id} + (\textrm{Id} - T_{1}^{-1}K)^{-1})T_{1}^{-1}(\textrm{Id} - KT_{1}^{-1}).
\end{equation*}
Thus $L_{V}$ can be written as
\begin{equation}
L_{V} = \chi_{V}R_{+}T_{1}^{-1}KT_{1}^{-1}(\textrm{Id} - KT_{1}^{-1})^{-1}J.
\end{equation}

We note that multiplication by $\chi_{V}$ to obtain $L_{V}$ from $L$ is bounded and hence preserves compactness. First we need to analyze $I_{\sigma,V}^{*}L_{V} = I_{\sigma,V}^{*}
\chi_{V}R_{+}T_{1}^{-1}KT_{1}^{-1}h$, where $h = h(x,\theta)$. Recall that
\begin{equation*}
[I_{\sigma,V}^{*}h](x)  = \int_{\mathbb{S}^{n-1}}\overline{E}(x,\theta)h^{\#}(x,\theta)\chi_{V}^{\#}(x,\theta)\,d\theta.
\end{equation*}
Again, as in \cite{inversesource} we notice that $\chi_{V}R_{+}T_{1}^{-1}g$ looks like $I_{\sigma,V}$, except that now the source depends on $\theta$ and $x$. Thus
\begin{align}
[I_{\sigma,V}^{*}\chi_{V}R_{+}T_{1}^{-1}g](x) & = \int_{\mathbb{S}^{n-1}}\overline{E}(x,\theta) \left[ \chi_{V}(x,\theta) \int_{-\infty}^{0}E(x+t\theta,\theta)g(x+t\theta,\theta)\,dt \right]^{\#}\,d\theta \notag \\
& = 2 \int_{\Omega_{1}} \frac{ \left[ \overline{E}\left( x, \frac{y-x}{|y-x|} \right) \chi_{V}^{\#}\left( y, \frac{y-x}{|y-x|} \right) E\left( y, \frac{y-x}{|y-x|} \right) g\left( y, \frac{y-x}{|y-x|} \right)\right]_{\mathrm{even}} }{|x-y|^{n-1}}\,dy, \label{term1}
\end{align}
where $F_{\mathrm{even}}(x,\theta)$ is the even part of $F$ with respect to $\theta$ (i.e. $F_{\mathrm{even}}(x, \theta) = \frac{1}{2}( F(x,\theta) + F(x,-\theta))$ ). To get back to $I_{\sigma,V}^{*}L_{V}$, we can let $g = KT_{1}^{-1}h$.

To proceed, we will now make a slightly weaker assumption on $k$ than stated in (\ref{regularity}) (see \cite{inversesource}). We will assume that $k$ can be written as the infinite sum
\begin{equation}
k(x,\theta,\theta') = \sum_{j=1}^{\infty} \Theta_{j}(\theta) \kappa_{j}(x,\theta') \label{kseries}
\end{equation}
where $\Theta_{j}$ and $\kappa_{j}$ are functions such that
\begin{equation}
\sum_{j=1}^{\infty} \|\Theta_{j}\|_{H^{1}(\mathbb{S}^{n-1})} \|\kappa_{j}\|_{L^{\infty}(\Omega_{1} \times \mathbb{S}^{n-1})} < \infty \label{convergence}
\end{equation}
In particular, we could take $\Theta_{j}$ to be the spherical harmonics $Y_{j}$, and then $\kappa_{j}$ would be the corresponding Fourier coefficients in such a basis. As discussed in \cite{inversesource}, uniform convergence of (\ref{kseries}) is guaranteed if $k \in L^{\infty}(\Omega_{1} \times \mathbb{S}_{\theta'}^{n-1}; C_{\theta}^{n+1}(\mathbb{S}^{n-1}))$, which is indeed a weaker assumption. 

Now let $K_{j}$ be the integral operator with kernel $\Theta_{j}\kappa_{j}$ and $B_{j} = \kappa_{j}T_{1}^{-1}$, where we regard $\kappa_{j}$ as integration in $\theta'$ against the kernel $\kappa_{j}$. Thus,
\begin{align}
& [K_{j}T_{1}^{-1}h](x,\theta) = \Theta_{j}(\theta)[B_{j}h](x),\\
& B_{j}h(x) = \int_{\Omega_{1}} \frac{\Sigma\left(x, |x-y|, \frac{x-y}{|x-y|}\right)\kappa_{j}\left( x, \frac{x-y}{|x-y|} \right)}{|x-y|^{n-1}} h\left(y, \frac{x-y}{|x-y|} \right)\,dy. \label{Bj}
\end{align}
By the proof of Lemma 1 in \cite{inversesource}, we have that $B_{j}( \textrm{Id} - KT_{1}^{-1})^{-1}J:L^{2}(\Omega_{1}) \to L^{2}(\Omega_{1})$ is compact. Now observe that
\begin{align}
\partial_{x}I_{\sigma,V}^{*}L_{V} & = \partial_{x}I_{\sigma,V}^{*} \chi_{V}R_{+}T_{1}^{-1}KT_{1}^{-1}(\textrm{Id} - KT_{1}^{-1})^{-1}J \notag \\
& = \sum_{j=1}^{\infty}[\partial_{x}I_{\sigma,V}^{*} \chi_{V}R_{+}T_{1}^{-1}\Theta_{j}J] \left[ B_{j}(\textrm{Id} - KT_{1}^{-1})^{-1}J \right] \label{series2}
\end{align}
By (\ref{term1}) and Proposition 1(b) of \cite{inversesource}, we have that $\partial_{x}I_{\sigma,V}^{*}\chi_{V}R_{+}T_{1}^{-1}\Theta_{j}J:L^{2}(\Omega_{1}) \to L^{2}(\Omega_{1})$ is bounded with a norm bounded above by $C \|\sigma\|_{C^{2}(\overline{\Omega} \times \mathbb{S}^{n-1})} \|\Theta_{j}\|_{H^{1}(\mathbb{S}^{n-1})}$. Thus each summand of (\ref{series2}) is a compact operator with norm bounded above by $C\|\Theta_{j}\|_{H^1} \|\kappa_{j}\|_{L^{\infty}}$, with $C$ depending on $\sigma$. By the condition (\ref{convergence}), we have that $\partial_{x}I_{\sigma,V}^{*}L_{V}$ is compact.

Now, the proof for $\partial_{x}L_{V}^{*}L_{V}$ is similar. In light of the fact that $B_{j}(\textrm{Id} - KT_{1}^{-1})^{-1}J$ is compact, it suffices to show that $\partial_{x}  L_{V}^{*}\chi_{V}R_{+}T_{1}^{-1}J$ is bounded. Note that $KT_{1}^{-1}$ commutes with $(\textrm{Id} - KT_{1}^{-1})^{-1}$, and hence
\begin{align}
L_{V}^{*}\chi_{V}R_{+}T_{1}^{-1}\Theta_{j}J & = \left( \chi_{V}R_{+}T_{1}^{-1}KT_{1}^{-1}(\textrm{Id} - KT_{1}^{-1})^{-1}J \right)^{*} \chi_{V}R_{+}T_{1}^{-1}\Theta_{j}J \label{L*L}\\
& = (KT_{1}^{-1}J)^{*}(\chi_{V}R_{+}T_{1}^{-1}(\textrm{Id} - KT_{1}^{-1})^{-1})^{*}\chi_{V}R_{+}T_{1}^{-1}J. \notag
\end{align}
As proven in \cite{inversesource}, by the boundedness of $\chi_{V}R_{+}T_{1}^{-1}$, the compactness of $\partial_{x}L_{V}^{*}\chi_{V}R_{+}T_{1}^{-1}J$ relies on $\partial_{x}(KT_{1}^{-1}J)^{*}$, and indeed it is.

Finally, to show that $\partial L_{V}^{*}I_{\sigma,V}$ is compact, we can proceed similarly to the case of $\partial L_{V}^{*}L_{V}$. Observe that $\partial_{x} L_{V}^{*}I_{\sigma,V} = L_{V}^{*}\chi_{V}R_{+}T_{1}^{-1}J$, which is equivalent to (\ref{L*L}) with $\Theta_{j} = 1$.\end{proof}

Now we are ready to prove the Theorem \ref{injectivetheorem} regarding the injectivity of $X_{V}$ when restricted to sources $f$ supported compactly in the visible set $\mathcal{M}$.
\begin{proof}[Proof of Theorem \ref{injectivetheorem}] Our proof mostly parallels the proof of Theorem 2 in \cite{inversesource}. By Lemma \ref{compact}, we have that $X_{V}^{*}X_{V}$ is equal to $I_{\sigma,V}^{*}I_{\sigma,V}$ plus a relative compact operator $\mathcal{L}_{V}$. First assume that $\sigma$ and $k$ are $C^{\infty}$. In this case, $I_{\sigma,V}^{*}I_{\sigma,V}$ is elliptic on $\mathcal{M}$, and thus there is a parametrix $Q$ of order $1$ which we view as an operator $Q: H^{1}(\Omega_{1}) \to L^{2}(\Omega)$ (We've restricted the image to $\Omega$, though $\mathcal{M}$ would do based on our assumption on the support of $f$). Thus, for $f$ supported in $W \Subset \mathcal{M}$, we have
\begin{equation}
QI_{\sigma,V}^{*}I_{\sigma,V}f = f + K_{1}f, \label{compact1}
\end{equation}
where $K_{1}$ is of order $-1$ near $\mathcal{M}$. Now apply $Q$ to $X_{V}^{*}X_{V}$ to get
\begin{equation}
QX_{V}^{*}X_{V}f = f + K_{1}f + Q\mathcal{L}_{V}f =: f + K_{2}f. \label{K2}
\end{equation}
By Lemma \ref{compact}, we have that $Q\mathcal{L}_{V}$ is compact. Furthermore, $K_{1}:L^{2}(\mathcal{M}) \to L^{2}(\mathcal{M})$ is compact by Rellich's lemma since it is smoothing near $\mathcal{M}$. This reduces the problem of inverting $X_{V}^{*}X_{V}$ to a Fredholm equation.
By Theorem \ref{Iinjective}, we have that for $\sigma$ real analytic on $\overline{\Omega}\times (\mathbb{S}^{n-1})$, $I_{\sigma,V}$ is injective when restricted to $f$ supported in $W \Subset \mathcal{M}$. 
From this point, the proof follows the same as that for Theorem 2 of \cite{inversesource}, and so we conclude.\end{proof}

\section{A Microlocal Result \label{partialwavefront}}
Although injectivity is a bit much to ask for in the partial data case, it is possible to analyze how singularities are propagated under the normal operator $I_{\sigma,V}^{*}I_{\sigma,V}$. Assuming suitable smoothness conditions on $k$, we will prove that one can partially recover the wavefront set of $f$. This is somewhat analogous to Proposition 1 in \cite{xraygeneric}, which allows one to partially recover the analytic wave front set of $f$ when $k = 0$.

To get more of an intuition for how to proceed, consider the case when $\|T_{1}^{-1}K\| < 1$. Since $X_{V} = \chi_{V}R_{+}T_{1}^{-1}(\textrm{Id} - KT_{1}^{-1})^{-1}J$, we may use a Neumann series expansion to get
\begin{equation}
X_{V} = \chi_{V}R_{+}T_{1}^{-1}\left( \sum_{j=0}^{\infty}(KT_{1}^{-1})^{j} \right)J.
\end{equation}
The term corresponding to $j=0$ is exactly $I_{\sigma,V}$, which does not account for scattering. Subsequent terms in the expansion incorporate scattering of higher and higher orders. So if we can show that $KT_{1}^{-1}$ is smoothing, then the most singular part of the data will be captured in the ballistic term. 

Before proceeding, we need to establish a bit of notation. We define the space $\mathcal{H}_{l}(\Omega \times \mathbb{S}^{n-1})$ as the completion of $C^{\infty}(\Omega \times \mathbb{S}^{n-1})$ with respect to the norm $\| \cdot \|_{\mathcal{H}_{l}(\Omega \times \mathbb{S}^{n-1})}$ given by
\begin{equation}
\|g(x,\theta)\|_{\mathcal{H}_{l}(\Omega \times \mathbb{S}^{n-1})} := \sum_{m=0}^{\infty}\sum_{k=1}^{k_{m,n}}\|a_{m}^{(k)}\|_{H^{l}(\Omega \times \mathbb{S}^{n-1})}\|Y_{m,n}^{(k)}\|_{H^{1}(\mathbb{S}^{n-1})},
\end{equation}
where $g(x,\theta) = \sum_{m=0}^{\infty}\sum_{k=1}^{k_{m,n}}a_{m}^{(k)}(x)Y_{m,n}^{(k)}(\theta)$ is the series representation of $g$ with respect to the spherical harmonics $Y_{m,n}^{(k)}(\theta)$ (see appendix \ref{singularresults}). Similarly, we define the space $\mathcal{C}_{l}(\Omega \times \mathbb{S}^{n-1})$ as the completion of $C^{\infty}(\Omega \times \mathbb{S}^{n-1})$ with respect to the norm $\| \cdot \|_{\mathcal{C}_{l}(\Omega \times \mathbb{S}^{n-1})}$ given by
\begin{equation}
\|g(x,\theta)\|_{\mathcal{C}_{l}(\Omega \times \mathbb{S}^{n-1})} := \sum_{m=0}^{\infty}\sum_{k=1}^{k_{m,n}}\|a_{m}^{(k)}\|_{C^{l}(\Omega)}\|Y_{m,n}^{(k)}\|_{H^{1}(\mathbb{S}^{n-1})}.
\end{equation}

The following lemma establishes the regularizing properties of the operator $KT_{1}^{-1}$.
\begin{lemma}Let $f \in \mathcal{H}_{l}(\Omega_{1} \times \mathbb{S}^{n-1})$ with $\textrm{supp}(f) \subseteq \Omega \times \mathbb{S}^{n-1}$ and $l\geq 1$, and suppose that $\sigma \in C^{\infty}(\overline{\Omega} \times \mathbb{S}^{n-1})$ and $k \in C^{\infty}(\overline{\Omega}_{x} \times \mathbb{S}_{\theta'}^{n-1} \times \mathbb{S}_{\theta}^{n-1})$. Then $KT_{1}^{-1}f \in \mathcal{H}_{l+1}(\Omega_{1} \times \mathbb{S}^{n-1})$. \label{reggain}\end{lemma}
\begin{proof}First let us recall from the proof of Lemma 1 in \cite{inversesource} that
\begin{equation}
[KT_{1}^{-1}f](x,\theta) = \int_{\Omega} \frac{\Sigma\left(x, |x-y|, \frac{x-y}{|x-y|}\right)k\left(x,\theta, \frac{x-y}{|x-y|}\right)}{|x-y|^{n-1}}f\left(y, \frac{x-y}{|x-y|}\right)\,dy,
\end{equation}
where $\Sigma(x,s,\theta') = \exp\left( -\int_{-s}^{0} \sigma(x + \tau\theta',\theta')\,d\tau\right)$. The characteristic $\Sigma(x,|x-y|,\theta')k(x,\theta, \frac{x-y}{|x-y|})$ satisfies the hypotheses of Proposition \ref{genreggain}, and the result follows.\end{proof}

\begin{corollary}Suppose that $\sigma \in C^{\infty}(\overline{\Omega} \times \mathbb{S}^{n-1})$ and $k \in C^{\infty}(\overline{\Omega}_{x} \times \mathbb{S}_{\theta'}^{n-1} \times \mathbb{S}_{\theta}^{n-1})$. Then $(KT_{1}^{-1})^{j}Jf:L^{2}(\Omega) \to \mathcal{H}_{j}(\Omega_{1} \times \mathbb{S}^{n-1})$ for all $j \geq 0$.\end{corollary}
From this result, we see that in the case that $\|KT_{1}^{-1}\| < 1$, $X_{V}f$ is equal to $I_{\sigma, V}f$ plus a remainder consisting of a series of terms with successively higher regularity, corresponding to higher order scattering.

\begin{lemma} Suppose that $\sigma \in C^{\infty}(\overline{\Omega} \times \mathbb{S}^{n-1})$ and $k \in C^{\infty}(\overline{\Omega} \times \mathbb{S}^{n-1} \times \mathbb{S}^{n-1})$. Then $KT_{1}^{-1}K:H^{l}(\Omega_{1} \times \mathbb{S}^{n-1}) \to H^{l}(\Omega_{1} \times \mathbb{S}^{n-1})$ is compact for all $l \geq 0$. \label{compactness}\end{lemma}

\begin{proof}Recall from the proof of Lemma 2 in \cite{inversesource} that
\begin{equation}
[KT_{1}^{-1}Kf](x,\theta) = \int \int_{\Omega_{1} \times \mathbb{S}^{n-1}} \frac{\alpha\left(x,y,|x-y|,\frac{x-y}{|x-y|}, \theta, \theta' \right)}{|x-y|^{n-1}}f(y,\theta')\,dy\,d\theta'
\end{equation}
with some $C^{\infty}$ $\alpha$ compactly supported in $x$ and $y$. The integral in $y$ is a weakly singular integral of the form in Proposition \ref{genreggain}, and so by part (a) we gain a derivative in $x$ for each fixed $\theta'$. Moreover, the smoothness in $\theta$ of $KT_{1}^{-1}Kf(x,\theta)$ is dependent only on the smoothness of $\alpha$. Therefore, $KT_{1}^{-1}K: H^{l}(\Omega_{1} \times \mathbb{S}^{n-1}) \to H^{l+1}(\Omega_{1} \times \mathbb{S}^{n-1})$. By Rellich's Lemma, the inclusion $H^{l+1}(\Omega_{1} \times \mathbb{S}^{n-1}) \hookrightarrow H^{l}(\Omega_{1} \times \mathbb{S}^{n-1})$ is compact, which completes the proof.\end{proof}

Again suppose that $\sigma$ and $k$ are $C^{\infty}$. Since $L_{V} =  \chi_{V}R_{+}T_{1}^{-1}KT_{1}^{-1}(\textrm{Id} - KT_{1}^{-1})^{-1}J$, we have by Lemma \ref{reggain} that $L_{V}:H^{l}(\Omega) \to \mathcal{H}_{l+1}(\Omega_{1} \times \mathbb{S}^{n-1})$. Since $X_{V}^{*}X_{V} = I_{\sigma,V}^{*}I_{\sigma,V} + \mathcal{L}_{V}$ and for smooth $\sigma$, $I_{\sigma,V}^{*}I_{\sigma,V}$ is a pseudodifferential operator of order $-1$, we would like to show that $\mathcal{L}_{V}$ maps $H^{l}(\Omega)$ into $H^{l+2}(\Omega_{1})$. We have the following proposition:
\begin{proposition}Let $l$ be a positive integer. There exists an open dense set $\mathcal{O}_{l}$ of pairs $(\sigma, k) \in C^{\infty}(\overline{\Omega} \times \mathbb{S}^{n-1}) \times C^{\infty}(\overline{\Omega}_{x} \times \mathbb{S}_{\theta'}^{n-1} \times S_{\theta}^{n-1})$ depending on $l$ such that for all $0 \leq l' \leq \frac{l}{2}$, the operator 
\begin{equation}
\mathcal{L}_{V} = I_{\sigma,V}^{*}L_{V} + L_{V}^{*}I_{\sigma,V} + L_{V}^{*}L_{V} = I_{\sigma,V}^{*}L_{V} + L_{V}^{*}X_{V}
\end{equation}
maps $H^{l'}(\Omega)$ into $H^{l'+2}(\Omega_{1})$. Moreover, we can write $\mathcal{L}_{V} = F + R$ where $F$ is a pseudodifferential operator of order $-2$, and $R:L^{2}(\Omega) \to H^{l}(\Omega)$. \label{normalremainderreggain}\end{proposition}
\begin{proof}First we write
\begin{align}
\mathcal{L}_{V}f & = I_{\sigma,V}^{*}L_{V}f + L_{V}^{*}X_{V}f \notag\\
& = \left( \chi_{V}R_{+}T_{1}^{-1}J\right)^{*}\chi_{V}R_{+}T_{1}^{-1}KT_{1}^{-1}(\textrm{Id} - KT_{1}^{-1})^{-1}Jf \notag \\
& \quad + \left( \chi_{V}R_{+}T_{1}^{-1}KT_{1}^{-1}(\textrm{Id} - KT_{1}^{-1})^{-1}J \right)^{*}\left(  \chi_{V}R_{+}T_{1}^{-1}(\textrm{Id} - KT_{1}^{-1})^{-1}J \right)f \label{LVexpression} \\
& =: I_{1}f + I_{2}f. \notag
\end{align}
Given equation (\ref{term1}), by Proposition \ref{genreggain} we have that $ \left( \chi_{V}R_{+}T_{1}^{-1}J\right)^{*}\chi_{V}R_{+}T_{1}^{-1}$ maps $\mathcal{H}_{l'+1}(\Omega_{1} \times \mathbb{S}^{n-1})$ into $H^{l'+2}(\Omega_{1})$ for all $l' \geq 0$.

We claim that for an open dense set of $(\sigma,k) \in C^{\infty}\times C^{\infty}$, $(\textrm{Id} - KT_{1}^{-1})^{-1}J$ maps $H^{l'}(\Omega_{1})$ to $\mathcal{H}_{l'}(\Omega_{1} \times \mathbb{S}^{n-1})$ for all $0 \leq l' \leq l+1$. First note that by Lemma \ref{compactness} $\lambda(KT_{1}^{-1})^{2}:H^{l'}(\Omega_{1} \times \mathbb{S}^{n-1}) \to H^{l'}(\Omega_{1} \times \mathbb{S}^{n-1})$ is compact for all $l' \geq 0$. Using the analytic Fredholm theorem on the resolvent $A(\lambda)$ with (\ref{resolvent2}) and (\ref{resolvent1}), we conclude that $(\textrm{Id} - \lambda KT_{1}^{-1})^{-1}$ exists and is bounded on $H^{l'}(\Omega_{1} \times \mathbb{S}^{n-1})$ for all $\lambda$ in some complex neighborhood of $[0,1]$ except for possibly a discrete set, which depends on $l'$. Taking the complement of the union of all such discrete sets for $0 \leq l' \leq l+1$, we obtain that $(\textrm{Id} - KT_{1}^{-1})^{-1}$ is bounded on each $H^{l'}(\Omega_{1} \times \mathbb{S}^{n-1})$ for all $0 \leq l' \leq l+1$ and for all but a discrete set of $\lambda$. So the set of pairs $(\sigma, k) \in C^{\infty} \times C^{\infty}$ for which (\ref{transport}) has a unique solution and $(\textrm{Id} - KT_{1}^{-1})^{-1}: H^{l'}(\Omega_{1} \times \mathbb{S}^{n-1}) \to H^{l'}(\Omega_{1} \times \mathbb{S}^{n-1})$ is bounded for $0 \leq l' \leq l+1$, is open and dense. Now, we just apply Lemma \ref{reggain} to the $KT_{1}^{-1}$ factor in $I_{1}$, which shows that $I_{1}$ maps $H^{l'}(\Omega)$ into $H^{l'+2}(\Omega_{1})$ for $0 \leq l' \leq l$.

To analyze $I_{2}$, we will use the series expansion of $k(x,\theta,\theta')$ previously defined in (\ref{kseries}). Observe that
\begin{align}
I_{2}f & = \left( \sum_{j=1}^{\infty} \chi_{V}R_{+}T_{1}^{-1}K_{j}T_{1}^{-1}(\textrm{Id} - KT_{1}^{-1})^{-1}J \right)^{*} \left( \chi_{V} R_{+}T_{1}^{-1}(\textrm{Id} - KT_{1}^{-1})^{-1}J \right) \notag\\
& = \left( \sum_{j=1}^{\infty} \left[\chi_{V}R_{+}T_{1}^{-1}\Theta_{j}J\right] \left[ B_{j}(\textrm{Id} - KT_{1}^{-1})^{-1}J\right] \right)^{*} \left( \chi_{V} R_{+}T_{1}^{-1}(\textrm{Id} - KT_{1}^{-1})^{-1}J\right) \notag\\
& = \sum_{j=1}^{\infty} \left( B_{j}(\textrm{Id} - KT_{1}^{-1})^{-1}J\right)^{*}\left( \chi_{V}R_{+}T_{1}^{-1}\Theta_{j}J \right)^{*}( \chi_{V} R_{+}T_{1}^{-1}(\textrm{Id} - KT_{1}^{-1})^{-1}J). \label{I2}
\end{align}
Similar to (\ref{term1}), we can compute
\begin{align}
& \left( \chi_{V}R_{+}T_{1}^{-1}\Theta_{j} J\right)^{*}\chi_{V} R_{+}T_{1}^{-1}g(x) \notag\\
 = \quad & 2 \int_{\Omega_{1}} \frac{ \left[ \overline{E}\left( x, \frac{y-x}{|y-x|} \right) \chi_{V}^{\#}\left( y, \frac{y-x}{|y-x|} \right) \Theta_{j}\left( \frac{y-x}{|y-x|}\right)E\left( y, \frac{y-x}{|y-x|} \right) g\left( y, \frac{y-x}{|y-x|} \right)\right]_{\mathrm{even}} }{|x-y|^{n-1}}\,dy.
\end{align}
By Proposition \ref{genreggain}, it is then evident that $\left( \chi_{V}R_{+}T_{1}^{-1}\Theta_{j} J\right)^{*}\chi_{V} R_{+}T_{1}^{-1}$ maps $\mathcal{H}_{l'}(\Omega \times \mathbb{S}^{n-1})$ into $H^{l'+1}(\Omega)$ for any $l' \geq 0$. Applying Proposition \ref{genreggain} to (\ref{Bj}) gives that $B_{j}(\textrm{Id} - KT_{1}^{-1})^{-1}J$ maps $H^{l'+1}(\Omega_{1})$ to $H^{l'+2}(\Omega_{1})$ for $0 \leq l' \leq l$. Therefore, so does its adjoint. Altogether, we have that $I_{2}f \in H^{l'+2}(\Omega_1)$ for $f \in H^{l'}(\Omega)$ where $0 \leq l' \leq l$.

For the next part of the proposition, note that for $m \geq 1$
\begin{equation*}
(\textrm{Id} - KT_{1}^{-1})\sum_{j=0}^{m}(KT_{1}^{-1})^{j}J = J -(KT_{1}^{-1})^{m+1}J.
\end{equation*}
By Lemma \ref{reggain} each summand $(KT_{1}^{-1})^{j}J$ is a pseudodifferential operator of order $-j$ with symbol depending smoothly on the parameter $\theta$. Therefore, we may construct a symbol
\begin{equation}
\rho(x,\xi,\theta) \sim \sum_{j=0}^{\infty}\sigma_{L}( (KT_{1}^{-1})^{j}J)(x,\xi,\theta), \label{asymptoticseries}
\end{equation}
where $\sigma_{L}: L_{1,0}^{-j}(\Omega) \to S^{-j}_{1,0}(\Omega \times \rn{n})$ is the full left symbol map. Here we are using the notation $L_{\delta, \rho}^{m}(\Omega)$ to refer to pseudodifferential operators with symbols in the class $S_{\delta, \rho}^{m}(\Omega \times \rn{n})$ (see \cite{grigis}). The symbol $\rho$ corresponds to a pseudodifferential operator $\widetilde{F}$ of order $0$ with smooth parameter $\theta$, and we have
\begin{align}
& (\textrm{Id} - KT_{1}^{-1}) \circ \widetilde{F} = J + \widetilde{R}_{0}, \quad \widetilde{R}_{0} \in (L_{1,0}^{-\infty}(\Omega); C^{\infty}(\mathbb{S}^{n-1})) \notag\\
& \widetilde{F} - \widetilde{R} = (\textrm{Id} - KT_{1}^{-1})^{-1}J \label{pseudoinverse} \\
 & \widetilde{R} = (\textrm{Id} - KT_{1}^{-1})^{-1}\widetilde{R}_{0}:L^{2}(\Omega) \to H^{l}(\Omega; C^{\infty}(\mathbb{S}^{n-1})).\notag
\end{align}
Substituting (\ref{pseudoinverse}) into the expression  (\ref{LVexpression}) for $\mathcal{L}_{V}$ gives that
\begin{align}
\mathcal{L}_{V} & = \Big[ \left( \chi_{V}R_{+}T_{1}^{-1}J\right)^{*}\chi_{V}R_{+}T_{1}^{-1}KT_{1}^{-1}\widetilde{F} +  \left( \chi_{V}R_{+}T_{1}^{-1}KT_{1}^{-1}\widetilde{F} \right)^{*}\left(  \chi_{V}R_{+}T_{1}^{-1}\widetilde{F} \right) \Big]+ R \notag\\
& =:  F + R, \label{LVpseudoremain}
\end{align}
where $R$ involves all the terms with $\widetilde{R}$. It is then immediate that $F$ is a pseudodifferential operator of order $-2$ while $R$ maps $L^{2}(\Omega) \to H^{l}(\Omega_{1})$.
\end{proof}

For reference, given an operator $A \in L_{1,0}^{m}(\Omega)$, we define $\textrm{WF}(A)$ as the smallest closed cone $\mathcal{C} \subset T^{*}\Omega \setminus 0$ such that $\sigma_{A} \vert_{\mathcal{C}^{c}} \in S^{-\infty}(\mathcal{C}^{c})$, where $\sigma_{A}$ is the symbol of $A$. We are now ready to prove the main theorem:
\begin{proof}[Proof of Theorem \ref{microtheorem}] Assume first that $(\sigma, k) \in C^{\infty} \times C^{\infty}$ is in the same open, dense set as in Proposition \ref{normalremainderreggain}. Since $\sigma$ is $C^{\infty}(\overline{\Omega} \times \mathbb{S}^{n-1})$, we have that $I_{\sigma,V}^{*}I_{\sigma,V}$ is a pseudodifferential operator of order $-1$. Furthermore, it is elliptic on $N^{*}l(x_{0},\theta_{0})$ by (\ref{microvisibleset}). Let $(z,\xi) \in N^{*}l(x_{0},\theta_{0})$. Then there exists a microlocal parametrix $Q \in L_{1,0}^{1}(\Omega_{1})$ elliptic at $(z,\xi)$ and $S_{1} \in L_{1,0}^{0}(\Omega)$ such that
\begin{equation*}
Q I_{\sigma,V}^{*}I_{\sigma,V} = \textrm{Id} + S_{1},
\end{equation*}
and $(z,\xi) \notin \textrm{WF}(S_{1})$. We will also restrict the image of $Q$ so that $Q:H^{1}(\Omega_{1}) \to L^{2}(\Omega)$. Since $\textrm{WF}(S_{1}f) \subset \textrm{WF}(S_{1}) \cap \textrm{WF}(f)$ (e.g. by Lemma 7.2 of \cite{grigis}), we have that $S_{1}f$ is microlocally smooth near $(z,\xi)$, i.e. $(z,\xi) \notin \textrm{WF}(S_{1}f)$.

Now we apply $Q$ to the normal operator $X_{V}^{*}X_{V} = I_{\sigma,V}^{*}I_{\sigma,V} + \mathcal{L}_{V}$ to get
\begin{equation*}
QX_{V}^{*}X_{V} = \textrm{Id} + S_{1} + Q\mathcal{L}_{V}.
\end{equation*}
By Proposition \ref{normalremainderreggain}, we have that $\mathcal{L}_{V} \in H^{k}(\Omega) \to H^{k+2}(\Omega_{1})$ for $0 \leq k \leq l$, and hence $Q\mathcal{L}_{V}: H^{k}(\Omega) \to H^{k+1}(\Omega)$. Moreover, from (\ref{LVpseudoremain}) we have that
\begin{equation*}
QX_{V}^{*}X_{V} = \textrm{Id} + QF + S_{1} + QR.
\end{equation*}
We can then construct $(\textrm{Id} + QF)^{-1}$, which is an elliptic pseudodifferential operator of order $0$ (its principal symbol is $1$). Thus
\begin{equation*}
(\textrm{Id} + QF)^{-1}QX_{V}^{*}X_{V} = \textrm{Id} + (\textrm{Id}+QF)^{-1}S_{1} +  (\textrm{Id}+QF)^{-1}QR.
\end{equation*}
We let 
\begin{equation}
v =  (\textrm{Id}+QF)^{-1}QRf \label{v}
\end{equation}
and note that $v \in H^{l}(\Omega)$. The other term $(\textrm{Id}+QF)^{-1}S_{1}f$ is microlocally smooth near $(z,\xi)$. Thus, $(z, \xi) \notin \textrm{WF}( (\textrm{Id} + QF)^{-1}QX_{V}^{*}X_{V}f) \Longrightarrow (z,\xi) \notin \textrm{WF}(f + v)$. Since
\begin{equation*}
\textrm{WF}((\textrm{Id} + QF)^{-1}QX_{V}^{*}X_{V}f) \subset \textrm{WF}((\textrm{Id} + QF)^{-1}Q) \cap \textrm{WF}(X_{V}^{*}X_{V}f),
\end{equation*}
the result follows.
\end{proof}

\begin{remark}It is easy to see that the open dense sets of pairs $(\sigma, k)$ that are dependent on $l$ form a nested sequence. Moreover, one can eliminate the remainder function $v$ by taking the intersection of all such open dense subsets. By the Baire category theorem, this limiting set of pairs will still be dense, but it is not clear if it remains open.\end{remark}

\begin{corollary}Suppose that $\|KT_{1}^{-1}\| < 1$. Then in Theorem \ref{microtheorem} we have $v = 0$.\end{corollary}
\begin{proof}In this case the series $\sum_{j=0}^{\infty}(KT_{1}^{-1})^{j}J$ converges to the identity plus a weakly singular integral operator $F$, which altogether is a pseudodifferential operator of order $0$. In light of (\ref{pseudoinverse}) this implies that $\widetilde{R_{0}} = 0$, and hence $R = 0$. By (\ref{LVpseudoremain}) and (\ref{v}) we have $v = 0$.\end{proof}

\begin{remark}
Here we make a brief mention of how one might utilize the above results to detect a source $f$ when partial data is known. Given $z \in \Omega$, suppose that one can measure $X_{V}\phi_{z}$ where $\phi_{z} = \delta(x - z)$ is a point source centered at $z$.

From the data $X_{V}f$, we may then compute $X_{V}^{*}X_{V}f(z)$, since
\begin{equation}
X_{V}^{*}X_{V}f(z) = \langle X_{V}^{*}X_{V}f, \phi_{z} \rangle_{L^{2}(\Omega_{1})} = \langle X_{V}f, X_{V}\phi_{z} \rangle_{L^{2}(\partial_{+}S\Omega_{1}, d\Sigma)}.
\end{equation}
By Theorem \ref{microtheorem}, the graph $\{ (z, X_{V}^{*}X_{V}f(z)) \, | \, z \in \Omega\}$ provides an image which indicates, modulo a function $v$ of some regularity depending on the regularity of $\sigma$ and $k$, the part of the wavefront set of $f$ contained in $\mathcal{M}'$.

However, in practice such a method of computing $X_{V}^{*}X_{V}$ might be too computationally expensive. So in the case that $\|T_{1}^{-1}K\|$ is suitably small, one could try using a Neumann series truncated to one or two terms to approximate $X_{V}^{*}$. The author intends to consider this further in future work.
\end{remark} 

\subsection*{Acknowledgements}This research was supported by NSF grant DMS-0838212. The results of this paper would not have been possible without some helpful and inspirational conversations with Gunther Uhlmann, and so for this the author expresses thanks. The author also thanks the journal referees for their helpful comments during the revision process.

\appendix

\section{Relevant Results on Singular Integral Operators \label{singularresults}}
\setcounter{section}{1}
Suppose we have an integral operator of the form
\begin{equation}
Kf(x) = a(x)f(x) +  \int K(x,x-y)f(y)\,dy, \qquad K(x,x-y) = r^{-n}\phi(x,\theta),
\end{equation}
where $f \in H^{l}(\Omega)$, $\theta = \frac{x-y}{|x-y|}$, and $r = |x-y|$. The function $\phi(x,\theta)$ is called the characteristic of the singular integral operator. We formally define the symbol $\Phi(x,\xi)$ of $K$ by
\begin{equation}
\Phi(x,\xi) = \int e^{-iz \cdot \xi}K(x,z)\,dz.
\end{equation}
It is easy to see by a change of variables that $\Phi$ is homogeneous of degree $0$ in $\xi$. Letting $\omega = \frac{\xi}{|\xi|}$, we will write $\Phi(x,\omega)$ from now on. It can be shown that if $K(x,x-y) = \phi(x,\theta)r^{-n}$, then
\begin{equation}
\Phi(x,\omega) = \int_{\mathbb{S}^{n-1}}\phi(x,\theta)\left[ \ln\left( \frac{1}{|\cos{\gamma}|} \right)+ \frac{i\pi}{2}\textrm{sign}\left( \cos{\gamma}\right) \right] \, d\theta
\end{equation}
where $\gamma$ is the angle between the vectors $x$ and $\omega$.

Consider the singular operator with a variable symbol,
\begin{equation}
(Af)(x):= a(x)f(x) + \int_{\rn{n}}\frac{\phi(x,\theta)}{|x-y|^{n}}f(y)\,dy = \int_{\rn{n}}e^{ix \cdot \xi} \Phi_{A}(x,\omega)\widehat{f}(\xi)\,d\xi, \quad \omega = \frac{\xi}{|\xi|}.\label{singularA}
\end{equation}
We introduce the class $\mathscr{R}_{l,\lambda}$ of those symbols that satisfy the condition
\begin{equation}
D_{x}^{\alpha}\Phi(x,\omega) \hat{\in} H^{\lambda}(\mathbb{S}^{n-1}), \quad \forall \alpha \, : \, |\alpha| \leq l.
\end{equation}
Here the relation $\beta(x,\omega) \hat{\in} H^{l}(\mathbb{S}^{n-1})$ means that
\begin{equation}
\int_{\mathbb{S}^{n-1}}|D_{\omega}^{\alpha}\beta(x,\omega)|^{2}\,d\omega \leq C, \quad 0 \leq | \alpha | \leq l.
\end{equation}
In this case, we say that $\beta(x,\omega)$ belongs to $H^{l}(\mathbb{S}^{n-1})$ uniformly with respect to the parameter $x$. For symbols of singular integral operators that satisfy such a condition, we have the following useful theorem.
\begin{theorem}[Theorem XI.9.2, \cite{mikhlin}] If $\Phi_{A}(x,\omega) \in \mathscr{R}_{l,\lambda}$ where $\lambda > \frac{n-1}{2}$, then the operator (\ref{singularA}) is bounded in $H^{l}(\rn{n})$.\label{9.2}\end{theorem}
For relating the characteristic $\phi(x,\theta)$ to its symbol $\Phi(x,\omega)$ we have the following theorem
\begin{theorem}[Theorem X.7.1, \cite{mikhlin}] The symbol of a singular integral operator satisfies the relation $\Phi(x,\omega) \hat{\in} H^{\lambda}(\mathbb{S}^{n-1})$ if and only if the characteristic of this integral satisfies the condition $\phi(x,\theta) \hat{\in} H^{l}(\mathbb{S}^{n-1})$ where $l = \lambda - \frac{n}{2}$.\label{mikhlin7.1}\end{theorem}

We also recall that the derivative of a weakly singular integral operator (\cite{mikhlin}, IX \S 7) is given by
\begin{equation}
\frac{\partial}{\partial x_{k}} \int_{\Omega} \frac{\phi(x,\theta)}{r^{n-1}}f(y)\,dy = \int_{\Omega}f(y) \frac{\partial}{\partial x_{k}}\left[ \frac{\phi(x,\theta)}{r^{n-1}} \right] \, dy - f(x) \int_{\mathbb{S}^{n-1}} \phi(x,\theta)\theta_{k}\,dS(y). \label{singularD}
\end{equation}
This formula holds for any $f \in L^{2}(\Omega)$ and for $\phi \in C^{1}(\Omega, \mathbb{S}^{n-1})$.

In dealing with weakly singular integral operators depending on a parameter $\theta$ and acting on functions $f$ depending on $x$ and $\theta'$, it will be helpful to work with a particular type of space on which these operators work nicely. In particular, we will use expansions of functions in terms of spherical harmonics. Recall that any function $g(x,\theta) \in C^{\infty}(\rn{n} \times \mathbb{S}^{n-1})$ can be expanded as a series
\begin{equation}
g(x,\theta) = \sum_{m=0}^{\infty}\sum_{k=1}^{k_{m,n}}a_{m}^{(k)}(x)Y_{m,n}^{(k)}(\theta),
\end{equation}
where
\begin{equation*}
k_{m,n} = \frac{(2m+n-2)(m+n-3)!}{(n-2)!m!}
\end{equation*}
denotes the number of linearly independent spherical functions of order $m$. Furthermore, if $g$ has compact support we claim that
\begin{equation}
\sum_{m=0}^{\infty}\sum_{k=1}^{k_{m,n}}\|a_{m}^{(k)}\|_{H^{l}(\rn{n})}\|Y_{m,n}^{(k)}\|_{H^{1}(\mathbb{S}^{n-1})} < \infty \quad \forall l \geq 0. \label{specialnormbound1}
\end{equation}
In \cite{inversesource} it is stated that for $g \in L^{\infty}(\rn{n}; C^{n+1}(\mathbb{S}^{n+1}))$ with compact support, we have that
\begin{equation}
\sum_{m=0}^{\infty}\sum_{k=1}^{k_{m,n}}\|a_{m}^{(k)}\|_{L^{\infty}(\rn{n})}\|Y_{m,n}^{(k)}\|_{H^{1}(\mathbb{S}^{n-1})} < \infty \label{specialnormbound2}
\end{equation}
Under the assumption that $g$ is compactly supported, we have that the $L^{\infty}$ norm is comparable to the $L^{2}$ norm. Since all derivatives of $g$ also satisfy (\ref{specialnormbound2}) with $\|a_{m}^{(k)}\|_{L^{\infty}(\rn{n})}$ replaced by $\|a_{m}^{(k)}\|_{L^{2}(\rn{n})}$, we have that (\ref{specialnormbound1}) holds.

Recall our definition of $\mathcal{H}_{l}(\Omega \times \mathbb{S}^{n-1})$ as the completion of $C^{\infty}(\Omega \times \mathbb{S}^{n-1})$ with respect to the norm $\| \cdot \|_{\mathcal{H}_{l}(\Omega \times \mathbb{S}^{n-1})}$, and $\mathcal{C}_{l}(\Omega \times \mathbb{S}^{n-1})$ as the completion of $C^{\infty}(\Omega \times \mathbb{S}^{n-1})$ with respect to the norm $\| \cdot \|_{\mathcal{C}_{l}(\Omega \times \mathbb{S}^{n-1})}$.

 The following proposition related to weakly singular integral operators and its applications in context will prove useful.
\begin{proposition}Let $A$ be the operator
\begin{equation*}
[Af](x) = \int \frac{ \alpha\left(x,y,|x-y|,\frac{x-y}{|x-y|}\right)}{|x-y|^{n-1}}f\left(y,\frac{x-y}{|x-y|} \right)\,dy
\end{equation*}
with $\alpha(x,y,r,\theta)$ compactly supported in $x$ and $y$. Then for a constant 
$C>0$ depending only on $n$ and $l$,
\begin{enumerate}
\item[a)] If $\alpha \in C^{2l+2}(\rn{n}_{x} \times \rn{n}_{y} \times \mathbb{R}_{r} \times \mathbb{S}^{n-1}_{\theta})$, then $A:H^{l}(\Omega) \to H^{l+1}(\rn{n})$ is continuous with a norm not exceeding $C\|\alpha\|_{C^{2l+2}}$.
\item[b)] If $\alpha(x,y,r,\theta) = \alpha'(x,y,r,\theta)\phi(\theta)$ and also in $C^{2l+2}$, then $$\|A\|_{H^{l}(\Omega) \to H^{l+1}(\rn{n})} \leq C \|\alpha'\|_{C^{2l+2}}\|\phi\|_{H^{1}(\mathbb{S}^{n-1})}.$$
\item[c)] If $f \in \mathcal{H}_{l}(\Omega)$ and $\alpha$ is as in (a), then $A:\mathcal{H}_{l}(\Omega \times \mathbb{S}^{n-1}) \to H^{l+1}(\Omega)$ is continuous with $\|A\|_{\mathcal{H}_{l}(\Omega \times \mathbb{S}^{n-1}) \to H^{l+1}(\rn{n})} \leq C \|\alpha\|_{C^{2l+2}}$.
\item[d)] If $\alpha = \alpha(x,y,r,\theta,\eta) \in C^{\infty}(\rn{n}_{x} \times \rn{n}_{y} \times \mathbb{R}_{r} \times \mathbb{S}^{n-1}_{\theta} \times \mathbb{S}^{n-1}_{\eta})$ is compactly supported in $x$ and $y$, then $A:\mathcal{H}_{l}(\Omega \times \mathbb{S}^{n-1}) \to \mathcal{H}_{l+1}(\rn{n} \times \mathbb{S}^{n-1})$ is bounded with $$\|A\|_{\mathcal{H}_{l}(\Omega \times \mathbb{S}^{n-1}) \to \mathcal{H}_{l+1}(\rn{n} \times \mathbb{S}^{n-1})} \leq C \|\alpha\|_{\mathcal{C}_{2l+2}(\rn{n} \times \mathbb{S}^{n-1})}.$$
\end{enumerate}
\label{genreggain}\end{proposition}

\begin{proof}First note that in all cases, since $\alpha$ is compactly supported in $x$ and $y$, we can let $U \Subset \rn{n}$ be such that $\alpha(x,y,r,\theta,\eta) = 0$ for $x,y \notin U$. We can then replace $\alpha$ by $\alpha \psi$ without affecting the integral, where $\psi \in C_{c}^{\infty}(\mathbb{R})$ satisfies $\psi(r) = 1$ for $r \in [0,\textrm{diam}(U)]$. Thus, we may as well assume that $\alpha$ is compactly supported in all variables.

Consider the case that $f$ is independent of $\theta$ and $\alpha$ only depends on $x$ and $\theta$. From the Calder\'on-Zygmund theory of singular operators, we know that for an integral operator $K$ with singular kernel $k(x,y) = \phi(x,\theta)r^{-n}$ where we recall $r = |x-y|$, if $\phi$ has mean value $0$ as a function of $\theta$ for each $x$, then $K$ is a well-defined operator on test functions, where the integration has to be understood in the principal value sense. Moreover, $K$ extends to a bounded operator on $L^{2}$ satisfying $\|K\| \leq C \sup_{x} \|\phi(x,\cdot)\|_{L^{2}(\mathbb{S}^{n-1})}$ (Theorem XI.3.1, \cite{mikhlin}). As a remark, the extension can be considered as a convolution in the sense of distributions, and then $\phi$ need not have mean value $0$ in $\theta$.

Let $(j_{1}, j_{2}, \ldots, j_{l+1})$ be a multi index. To make notation a bit more consistent, let $\alpha_{0} = \alpha$. Consider the derivative $\partial_{x_{j_{1}}}A$, which by (\ref{singularD}) and (\cite{inversesource}, Proposition 1) consists of a bounded term $a_{1}(x)f(x)$ plus the integral operator with kernel
\begin{equation} 
\partial_{x_{j_{1}}} \frac{\alpha(x,\theta)}{r^{n-1}} =\frac{(1-n)\theta_{j_{1}}\alpha + \partial_{\theta_{j_{1}}}\alpha}{r^{n}} + \frac{\partial_{x_{j_{1}}}\alpha(x,\theta)}{r^{n-1}}. \label{prop1eq1}
\end{equation}
Letting $\phi_{1}(x,\theta) := (1-n)\theta_{j_{1}}\alpha(x,\theta) + \partial_{\theta_{j_{1}}}\alpha(x,\theta)$, since $\alpha \in C^{2l+2}$ and is compactly supported in $x$, we have that the symbol
\begin{equation*}
\Phi_{1}(x,\omega) = \int e^{-i z \cdot \omega}\frac{\phi_{1}\left(x,\frac{z}{|z|}\right)}{|z|^{n}}\,dz
\end{equation*}
of $\phi_{1}$ belongs to $C^{2l+1}(\rn{n} \times \mathbb{S}^{n-1})$. Since $\phi_{1}$ is compactly supported in $x$,  $\partial_{x}^{\gamma} \phi_{1} \hat{\in} H^{2l+1-|\gamma|}(\mathbb{S}^{n-1}) \subset H^{1+\lambda - \frac{n}{2}}(\mathbb{S}^{n-1})$ for all $0 \leq |\gamma| \leq l$ and for some fixed $\lambda > \frac{n-1}{2}$ (in particular we could take $\lambda = \frac{n}{2}$). By Theorem \ref{mikhlin7.1}, $\Phi_{1}(x,\omega) \in \mathscr{R}_{l,\lambda}$. By Theorem \ref{9.2} we have that the integral kernel $\phi_{1}(x,\theta)r^{-n}$ corresponds to a singular integral operator that is bounded on $H^{l}$. For the second term in (\ref{prop1eq1}), which is a weakly singular integral kernel, we have that $\alpha_{1}(x,\theta) := \partial_{x_{j_{1}}}\alpha(x,\theta) \in C^{2l+1}$. Similarly as before, we compute $\partial_{x_{j_{2}}}\left( \frac{\alpha_{1}(x,\theta)}{r^{n-1}}\right)$ which corresponds to an operator with a bounded multiplier $a_{2}(x)$, a singular integral operator, and a weakly singular integral operator. It can be shown analogously that the symbol $\Phi_{2}(x,\xi)$ corresponding to the characteristic $\phi_{2}(x,\theta)$ of the singular integral term belongs to $\mathscr{R}_{l-1,\lambda}$. Thus modulo a weakly singular integral operator, the operator $A_{2}$ with kernel $\phi_{2}(x,\theta)r^{-n}$ is bounded on $H^{l-1}$. We then focus our attention on the weakly singular integral operator that remains.

After repeating this process a total of $l+1$ times, which involves $l+1$ differentiations, the remaining weakly singular integral operator has a kernel $\alpha_{l+1}(x,\theta)r^{-n+1}$ with $\alpha_{l+1}(x,\theta) \in C^{l+1}$. We can then proceed as in the proof of  (\cite{inversesource}, Proposition 1) to obtain that this term is bounded on $L^{2}(\Omega)$. In particular, we use the criterion from Calder\'on Zygmund Theory which states that if $K$ is an integral operator with integral kernel $k(x,y)$ satisfying
\begin{equation}
\sup_{x} \int |k(x,y)|\,dy \leq M, \quad \sup_{y} \int |k(x,y)|\,dx \leq M, \label{zygmund}
\end{equation}
then $K$ is bounded in $L^{2}$ with a norm not exceeding $M$ (\cite{stein}, Prop. A.5.1).

Now we want to bound the operator norm $\|A\|_{H^{l}(\Omega) \to H^{l+1}(\rn{n})}$ in this simpler case. Let $\phi_{i}$ be the characteristic of the $i$th singular integral operator obtained by the above process with symbol $\Phi_{i}$ and $\widetilde{A}_{i}$ the corresponding singular integral operator. Note that $\widetilde{A}_{i}$ is bounded on $H^{l-i+1}(\Omega)$. One can explicitly compute that
\begin{align}
\alpha_{i} & = \partial_{x_{j_{1}}x_{j_{2}}\cdots x_{j_{i}}}\alpha,\\
\phi_{i} & = (1-n)\theta_{j_{i}} \alpha_{i-1} + \partial_{\theta_{j_{i}}} \alpha_{i-1},\\
a_{i}(x) & = \int \alpha_{i-1}(x,\theta)\theta_{j_{i}}\, dS(y), \quad \textrm{for } 1 \leq i \leq l+1.
\end{align}
Also define
\begin{equation}
\widetilde{R}_{i}f(x) := \int \frac{\alpha_{i}(x,\theta)}{r^{n-1}}f(y)\,dy.
\end{equation}
We have
\begin{align}
\|\phi_{i}\|_{C^{l+1-i}} & \leq C\|\alpha\|_{C^{l+1}} \notag\\
\|\alpha_{i}\|_{C^{l+1-i}} & \leq C\|\alpha\|_{C^{l+1}} \notag\\
\|a_{i}\|_{C^{l+1-i}} & \leq C\|\alpha_{i-1}\|_{C^{l+1-i}} \leq C\|\alpha_{i-1}\|_{C^{l+1- (i-1)}} \leq C'\|\alpha\|_{C^{l+1}} \label{kernelbound}
\end{align}
If $\|f\|_{H^{l}(\Omega)} = 1$, then (\cite{mikhlin}, Thm XI.3.2, Thm XI.9.2) imply that for $1 \leq i \leq l+1$,
\begin{align}
\|\widetilde{A}_{i}f\|_{H^{l-i+1}(\rn{n})} & = \sum_{|\beta| \leq l-i+1}\|D_{x}^{\beta}\widetilde{A}_{i}f\|_{L^{2}(\rn{n})} \notag\\
& \leq \sum_{|\beta| \leq l-i+1}\sum_{\gamma \leq \beta} \frac{\beta !}{\gamma ! (\beta - \gamma)!} \sup_{x} \|D_{x}^{\beta - \gamma}\Phi_{i}(x,\cdot)\|_{H^{\frac{n}{2}}(\mathbb{S}^{n-1})} \|f\|_{H^{|\beta|}(\Omega)} \notag \\
& \leq C\sum_{|\beta| \leq l-i+1}\sum_{\gamma \leq \beta} \sup_{x} \|D_{x}^{\beta - \gamma}\phi_{i}(x,\cdot)\|_{L^{2}(\mathbb{S}^{n-1})} \|f\|_{H^{l}(\Omega)} \notag\\
& \leq C \|\phi_{i}\|_{C^{l-i+1}} \notag\\
& \leq C \|\alpha\|_{C^{l+1}}. \label{tildeAibound}
\end{align}
Now we estimate $\|Af\|_{H^{l+1}}$ using (\ref{tildeAibound}) with the understanding that we sum over all indices $j_{1},j_{2},\ldots$. Again, assume that $\|f\|_{H^{l}(\Omega)} = 1$.
\begin{align*}
\|Af\|_{H^{l+1}} & = \|Af\|_{L^{2}} + \|\partial_{x_{j_{1}}}Af\|_{H^{l}}\\
& \lesssim \|\alpha\|_{C^{0}} + \|a_{1}f\|_{H^{l}} + \|\widetilde{A}_{1}f\|_{H^{l}} + \|\widetilde{R}_{1}f\|_{H^{l}}\\
& \lesssim\|\alpha\|_{C^{0}} + \|a_{1}\|_{C^{l}} + \|\alpha\|_{C^{l+1}} + \|\widetilde{R}_{1}f\|_{L^{2}} + \|\partial_{x_{j_{2}}}\widetilde{R}_{1}f\|_{H^{l-1}}\\
& \lesssim \|\alpha\|_{C^{l+1}} + \|\alpha\|_{C^{0}} + \|\alpha_{1}\|_{C^{0}} + \|a_{1}\|_{C^{l}} + \|a_{2}f\|_{H^{l-1}} + \|\widetilde{A}_{2}f\|_{H^{l-1}} + \|\widetilde{R}_{2}f\|_{H^{l-1}}\\
& \lesssim \|\alpha\|_{C^{l+1}} +  \|\alpha\|_{C^{0}} + \|\alpha_{1}\|_{C^{0}} + \|a_{1}\|_{C^{l}} + \|a_{2}\|_{C^{l-1}} + \|\widetilde{R}_{2}f\|_{H^{l-1}}\\
& \vdots\\
& \lesssim \|\widetilde{R}_{l+1}\|_{L^{2}} + \|\alpha\|_{C^{l+1}} + \sum_{i=1}^{l+1}\|\alpha_{i-1}\|_{C^{0}} + \|a_{i}\|_{C^{l+1-i}}\\
& \lesssim \|\alpha_{l+1}\|_{C^{0}} + \|\alpha\|_{C^{l+1}} + \sum_{i=1}^{l+1}\|\alpha\|_{C^{l+1}}\\
& \lesssim \|\alpha\|_{C^{l+1}}.
\end{align*}
Thus in the simplified case where $\alpha$ only depends on $x$ and $\theta$ and $f$ is independent of $\theta$, we have
\begin{equation}
\|A\|_{H^{l}(\Omega) \to H^{l+1}(\rn{n})} \leq C \|\alpha\|_{C^{l+1}}.
\end{equation}

To extend to $\alpha = \alpha_{0}$ depending also on $y$ and $r$, we use a first order Taylor expansion in $y$ and $r$ centered at $y=x$ and $r = 0$, similarly to in (\cite{inversesource}, Prop. 1), to get
\begin{align*}
\alpha_{0}(x,y,r,\theta) & = \alpha_{0}(x,x,0,\theta) + \sum_{|\beta| + |\gamma| = 1}r^{|\beta|}(y-x)^{\gamma}\int_{0}^{1}\partial_{r}^{\beta}\partial_{y}^{\gamma}\alpha_{0}(x,x+t(y-x),tr, \theta)\,dt\\
& = \alpha_{0}(x,x,0,\theta) + \sum_{|\beta| + |\gamma| = 1}r^{|\beta|}(-r\theta)^{\gamma}\int_{0}^{1}\partial_{r}^{\beta}\partial_{y}^{\gamma}\alpha_{0}(x,x+t(y-x),tr, \theta)\,dt\\
& = \alpha_{0}(x,x,0,\theta) + r\sum_{|\beta| + |\gamma| = 1}(-1)^{|\gamma|}\theta^{\gamma}\int_{0}^{1}\partial_{r}^{\beta}\partial_{y}^{\gamma}\alpha_{0}(x,x+t(y-x),tr, \theta)\,dt.
\end{align*}
We can then write
\begin{equation}
\alpha_{0}(x,y,r,\theta) = \alpha_{0}(x,x,0,\theta) + r \gamma_{1}(x,y,r,\theta) \label{prop1eq2}
\end{equation}
where $\gamma_{1} \in C^{2l+1}$. After dividing by $r^{n-1}$, the first term in (\ref{prop1eq2}) maps $H^{l}$ to $H^{l+1}$ by the previous argument. The second term corresponds to an integral operator with kernel $\gamma_{1}(x,y,r,\theta)r^{-n+2}$. If we differentiate this with respect to $x$, we get a weakly singular integral operator with kernel $\alpha_{1}(x,y,r,\theta)r^{-n+1}$ where $\alpha_{1} \in  C^{2l}$. Now repeat as before, writing
\begin{equation}
\alpha_{1}(x,y,r,\theta) = \alpha_{1}(x,x,0,\theta) + r \gamma_{2}(x,y,r,\theta) \label{prop1eq3}
\end{equation}
where $\gamma_{2} \in C^{2l-1}$. The first term $\alpha_{1}(x,x,0,\theta)$ corresponds to a bounded operator $A_{1}:H^{l-1}(\Omega) \to H^{l}(\rn{n})$. Moreover, $\gamma_{2}(x,y,r,\theta)r^{-n+2}$ can be differentiated with respect to $x$ to obtain a weakly singular integral operator with kernel $\alpha_{2}(x,y,r,\theta)r^{-n+1}$ where $\alpha_{2} \in C^{2l-2}$.

After repeating this process a total of $l$ times, we have a remainder term that is a weakly singular integral operator with kernel $\alpha_{l}(x,y,r,\theta)r^{-n+1}$ where $\alpha_{l} \in C^{2}$. Write
\begin{equation*}
\alpha_{l}(x,y,r,\theta) = \alpha_{l}(x,x,0,\theta) + r\gamma_{l+1}(x,y,r,\theta)
\end{equation*}
with $\gamma_{l+1} \in  C^{1}$. Then $\gamma_{l+1}$ corresponds to the operator $\gamma_{l+1}(x,y,r,\theta)r^{-n+2}$, which we can differentiate with respect to $x$ to obtain a weakly singular operator that is bounded on $L^{2}$ with a bound not exceeding $\|\gamma_{l+1}\|_{C^{1}}$ by using the estimates in (\ref{zygmund}) and applying the Calder\'on Zygmund theorem. Since each weakly singular integral operator with kernel $\alpha_{j}(x,x,0,\theta)r^{-n+1}$ is a bounded map from $H^{l-j}(\Omega) \to H^{l-j+1}(\rn{n})$, we combine the remainder terms together to get that $A:H^{l}(\Omega) \to H^{l+1}(\rn{n})$. 

More explicitly, let $A_{i}$ be the weakly singular integral operator with kernel $\alpha_{i}(x,x,0,\theta)r^{-n+1}$ and $R_{i}$ the integral operator with kernel $\gamma_{i+1}(x,y,r,\theta)r^{-n+2}$. In particular, $A_{i} = \partial_{x_{j_{i}}}R_{i-1}$. We will also need the straightforward estimates
\begin{equation}
\|\alpha_{i}\|_{C^{l-i+1}} \lesssim \|\alpha\|_{C^{l+i+1}}, \quad \|\gamma_{i}\|_{C^{m}} \lesssim \|\alpha\|_{C^{m+2i-1}}.
\end{equation}
For $\|f\|_{H^{l}} = 1$ we have
\begin{align*}
\|Af\|_{H^{l+1}} & \leq \| A_{0}f\|_{H^{l+1}} + \| R_{0}f\|_{H^{l+1}}\\
& \leq \|\alpha_{0}\|_{C^{l+1}} + \|R_{0}f\|_{L^{2}} + \|\partial_{x_{j_{1}}}R_{0}f\|_{H^{l}}\\
& \lesssim \|\alpha_{0}\|_{C^{l+1}} + \|\gamma_{1}\|_{C^{0}} + \|A_{1}f\|_{H^{l}} + \|R_{1}f\|_{H^{l}}\\
& \lesssim \|\alpha_{0}\|_{C^{l+1}} +  \|\gamma_{1}\|_{C^{0}} + \|\alpha_{1}\|_{C^{l}} + \|R_{1}f\|_{L^{2}} + \|\partial_{x_{j_{2}}}R_{1}f\|_{H^{l-1}}\\
& \lesssim \|\alpha_{0}\|_{C^{l+1}} +  \|\alpha_{1}\|_{C^{l}}  + \|\gamma_{1}\|_{C^{0}} + \|\gamma_{2}\|_{C^{0}} + \|A_{2}f\|_{H^{l-1}} + \|R_{2}f\|_{H^{l-1}}\\
& \vdots\\
& \lesssim \|A_{l}f\|_{H^{1}} + \|R_{l}f\|_{H^{1}} + \sum_{i=0}^{l-1} \|\alpha_{i}\|_{C^{l+1-i}} + \|\gamma_{i+1}\|_{C^{0}}\\
& \lesssim \|\alpha_{l}\|_{C^{1}} + \|R_{l}f\|_{L^2} + \|\partial_{x_{j_{l+1}}}R_{l}f\|_{L^{2}} +  \sum_{i=0}^{l-1} \|\alpha_{i}\|_{C^{l+1-i}} + \|\gamma_{i+1}\|_{C^{0}}\\
& \lesssim \|\alpha\|_{C^{2l+1}} + \|\gamma_{l+1}\|_{C^{0}} + \|\gamma_{l+1}\|_{C^{1}} + \sum_{i=0}^{l-1} \|\alpha\|_{C^{l+1+i}} + \|\alpha\|_{C^{2i+1}}\\
& \lesssim \|\alpha\|_{C^{2l+1}} + \|\alpha\|_{C^{2l+2}} +  \sum_{i=0}^{l-1} \|\alpha\|_{C^{l+1+i}} + \|\alpha\|_{C^{2i+1}}\\
& \lesssim \|\alpha\|_{C^{2l+2}}.
\end{align*}
This proves (a).

Now consider if $\alpha(x,y,r,\theta) = \alpha'(x,y,r,\theta)\phi(\theta)$. Then
\begin{equation}
(1-n)\theta_{j}\alpha + \partial_{\theta_{j}}\alpha = (1-n)\theta_{j}\alpha' \phi + \alpha' \partial_{\theta_{j}}\phi + \phi \partial_{\theta_{j}}\alpha'.
\end{equation}
In short, for each term in the decomposition of $A$ by differentiation, $\phi$ is differentiated exactly once. Therefore
\begin{equation}
\|A\|_{H^{l}(\Omega) \to H^{l+1}(\rn{n})} \leq C \|\alpha'\|_{C^{2l+2}}\|\phi|_{H^{1}(\mathbb{S}^{n-1})},
\end{equation}
which proves (b).

For $f \in \mathcal{H}_{l}(\Omega \times \mathbb{S}^{n-1})$ depending on $\theta$ as well, we expand $f$ as a series
\begin{equation*}
f(x,\theta) = \sum_{m=0}^{\infty}\sum_{k=1}^{k_{m,n}}a_{m}^{(k)}(x)Y_{m,n}^{(k)}(\theta),
\end{equation*}
Then
\begin{align*}
[Af](x) & = \int \frac{\alpha(x,y,r,\theta)}{r^{n-1}}f(y,\theta)\,dy\\
& =  \int \frac{\alpha(x,y,r,\theta)}{r^{n-1}}\sum_{m=0}^{\infty}\sum_{k=1}^{k_{m,n}}a_{m}^{(k)}(y)Y_{m,n}^{(k)}(\theta)\,dy\\
& = \sum_{m=0}^{\infty}\sum_{k=1}^{k_{m,n}} \int \frac{\alpha(x,y,r,\theta)Y_{m,n}^{(k)}(\theta)}{r^{n-1}}a_{m}^{(k)}(y)\,dy.
\end{align*}
Hence
\begin{align*}
\|Af\|_{H^{l+1}(\rn{n} \times \mathbb{S}^{n-1})} & \leq \sum_{m=0}^{\infty}\sum_{k=1}^{k_{m,n}} \left\|  \int \frac{\alpha(x,y,r,\theta)Y_{m,n}^{(k)}(\theta)}{r^{n-1}}a_{m}^{(k)}(y)\,dy \right\|_{H^{l+1}(\rn{n} \times \mathbb{S}^{n-1})}\\
& \leq  \sum_{m=0}^{\infty}\sum_{k=1}^{k_{m,n}} \|A_{m,n}^{(k)}\|_{H^{l} \to H^{l+1}}\|a_{m}^{(k)}\|_{H^{l}(\Omega)}\\
& \leq C\|\alpha\|_{C^{2l+2}} \sum_{m=0}^{\infty}\sum_{k=1}^{k_{m,n}} \|Y_{m,n}^{(k)}\|_{H^{1}(\mathbb{S}^{n-1})}\|a_{m}^{(k)}\|_{H^{l}(\Omega)}\\
& \leq C\|\alpha\|_{C^{2l+2}}\|f\|_{\mathcal{H}_{l}(\Omega)}.
\end{align*}
Here the operator $A_{m,n}^{(k)}$ is given by $[A_{m,n}^{(k)}g](x) = \int \frac{\alpha(x,y,r,\theta)Y_{m,n}^{(k)}(\theta)}{r^{n-1}}g(y)\,dy$. This proves (c).

Finally, if $\alpha = \alpha(x,y,r,\theta,\eta)$ is $C^{\infty}$ with compact support, we can expand it as a series
\begin{equation}
\alpha(x,y,r,\theta,\eta) = \sum_{m=0}^{\infty}\sum_{k=1}^{k_{m,n}}b_{m}^{(k)}(x,y,r,\theta)Y_{m,n}^{(k)}(\eta).
\end{equation}
Note that $\alpha \in \mathcal{C}_{j}(\rn{n} \times \rn{n} \times \mathbb{R} \times \mathbb{S}^{n-1} \times \mathbb{S}^{n-1})$ for all $j \geq 0$. Then for $f \in \mathcal{H}_{l}(\Omega \times \mathbb{S}^{n-1})$, we have
\begin{align*}
& \qquad \|Af\|_{\mathcal{H}_{l+1}(\rn{n} \times \mathbb{S}^{n-1})}\\
& = \sum_{m_{1}=0}^{\infty}\sum_{k_{1}=1}^{k_{m_{1},n}}\|Y_{m_{1},n}^{(k_{1})}\|_{H^{1}(\mathbb{S}^{n-1})}\left\| \int \frac{b_{m_{1}}^{(k_{1})}(x,y,r,\theta)}{r^{n-1}}f(y,\theta)\,dy \right\|_{H^{l+1}(\rn{n} \times \mathbb{S}^{n-1})}\\
& = \sum_{m_{1}=0}^{\infty}\sum_{k_{1}=1}^{k_{m_{1},n}}\|Y_{m_{1},n}^{(k_{1})}\|_{H^{1}(\mathbb{S}^{n-1})} \left\| \sum_{m_{2}=0}^{\infty}\sum_{k_{2}=1}^{k_{m_{2},n}}\int \frac{b_{m_{1}}^{(k_{1})}(x,y,r,\theta)Y_{m_{2},n}^{(k_{2})}(\theta)}{r^{n-1}}a_{m_{2}}^{k_{2}}(y)\,dy \right\|_{H^{l+1}(\rn{n} \times \mathbb{S}^{n-1})}\\
& \lesssim  \sum_{m_{1}=0}^{\infty}\sum_{k_{1}=1}^{k_{m_{1},n}}\sum_{m_{2}=0}^{\infty}\sum_{k_{2}=1}^{k_{m_{2},n}}\|Y_{m_{1},n}^{(k_{1})}\|_{H^{1}(\mathbb{S}^{n-1})}\| b_{m_{1}}^{(k_{1})}\|_{C^{2l+2}}\|Y_{m_{2},n}^{(k_{2})}\|_{H^{1}(\mathbb{S}^{n-1})}\|a_{m_{2}}^{k_{2}}\|_{H^{l}(\Omega)}\\
& = \|\alpha\|_{\mathcal{C}_{2l+2}}\|f\|_{\mathcal{H}_{l}(\Omega \times \mathbb{S}^{n-1})}.
\end{align*}
\end{proof}

\end{document}